\documentclass[11pt,a4paper]{article}

\usepackage[utf8]{inputenc}
\usepackage[T1]{fontenc}
\usepackage{amsmath,amsthm,amsopn,amsfonts,amssymb,dsfont,color}
\usepackage{mathrsfs}
\usepackage{enumitem}
\usepackage{xcolor}

\usepackage{hyperref}

\usepackage{a4,a4wide}

\def\ep{{\varepsilon}}
\let\epsilon=\varepsilon
\def\R{\mathbb R}
\def\C{\mathbb C}
\def\PP{\mathbb P}
\def\dans{\longrightarrow}
\def\weak{\rightharpoonup}

\newcommand{\E}{\mathscr E}
\newcommand{\s}{\mathrm s}
\newcommand{\ess}{\mathrm{ess}}
\renewcommand{\Re}{\mathrm{Re}\,}
\newcommand{\1}{\mathbf 1}
\newcommand{\TV}{\mathrm{TV}}
\DeclareMathOperator*{\essinf}{ess\,inf}
\DeclareMathOperator*{\esssup}{ess\,sup}
\DeclareMathOperator*{\meas}{meas}

\newcommand\LL{\mathscr{L}}
\newcommand\MM{\mathscr{M}}
\newcommand\A{A} 
\newcommand\BB{B} 
\let\phi=\varphi

\newtheorem{theorem}{\textbf{Theorem}}[section]

\newtheorem{lemma}[theorem]{\textbf{Lemma}}
\newtheorem{proposition}[theorem]{\textbf{Proposition}}

\newtheorem{corollary}[theorem]{\textbf{Corollary}}
\newtheorem{definition}[theorem]{\textbf{Definition}}

\newtheorem{assumption}{\textbf{Assumption}}[section]
\theoremstyle{remark}

\newtheorem{example}[theorem]{\textbf{Example}}

\numberwithin{equation}{section}

\title{Confining integro-differential equations 
originating from evolutionary biology: 
ground states and long time dynamics}
\date{}

\begin{document}

\maketitle

\begin{center}
{\large\bf Matthieu Alfaro \footnote{Universit\'e de Rouen Normandie, CNRS, Laboratoire de Math\'ematiques Rapha\"el Salem, Saint-Etienne-du-Rouvray, France \& BioSP, INRAE, 84914, Avignon, France. E-mail: matthieu.alfaro@univ-rouen.fr}, Pierre Gabriel \footnote{\label{uvsq}Laboratoire de Mathématiques de Versailles, UVSQ, CNRS, Université Paris-Saclay, 45 Avenue des États-Unis, 78035Versailles cedex, France. Email: pierre.gabriel@uvsq.fr} and Otared Kavian \footnote{Laboratoire de Mathématiques de Versailles, UVSQ, CNRS, Université Paris-Saclay, 45 Avenue des États-Unis, 78035Versailles cedex, France. Email: otared.kavian@uvsq.fr}}\\
[2ex]

\end{center}


\tableofcontents

\vspace{10pt}

\begin{abstract} 
We consider nonlinear mutation selection models, known as {\it replicator-mutator} equations in evolutionary biology. They involve a nonlocal mutation kernel and a confining fitness potential. We prove that the long time behaviour of the Cauchy problem is determined by the principal eigenelement of the underlying linear operator.
The novelties compared to the literature on these models are about the case of symmetric mutations: we propose a new milder sufficient condition for the existence of a principal eigenfunction, and we provide what is to our knowledge the first quantification of the spectral gap. We also recover existing results in the non-symmetric case, through a new approach.
\\

\noindent{\underline{Key Words:} evolutionary genetics, nonlocal diffusion, eigenelements, long time behaviour.}\\

\noindent{\underline{AMS Subject Classifications:} 45K05 (Integro partial differential equations), 92D15 (Problems related to evolution), 45C05 (Eigenvalue problems), 35B40 (Asymptotic behavior of solutions).}
\end{abstract}

\section{Introduction}\label{sec:intro}

The starting point of this work is the nonlinear integro-differential equation
\begin{equation}
\label{eq:rep-mut}
\partial_{t}u = \sigma ^2(J*u-u)-\left(W(x)-\overline W[u](t)\right)u, \quad t>0,\; x\in \R^N,
\end{equation}
where $\sigma>0$ is a given parameter, $J : x \mapsto J(x)$ a probability density on $\R^N$, and $W : x \mapsto W(x)$ a so-called confining potential (meaning that $W(x)\to +\infty$ as $|x| \to +\infty$), while ${\overline W}[u](t)$ is a nonlocal term defined by
$$ {\overline W}[u](t) := \langle W, u(t,\cdot)\rangle := \int_{\R^{N}} W(y)u(t,y)dy,$$
and can be seen, at least formally, as a Lagrange multiplier: it ensures that the solution to \eqref{eq:rep-mut} starting at $t=0$ from a probability density $u_{0}(x) = u(0,x)$ remains a probability density for $t>0$.

This model is known as a {\it replicator-mutator} model in evolutionary biology. In this context, at time $t>0$, $u(t,\cdot)$ stands for the probability distribution of the phenotypic trait $x$ (in a population) on the multi-dimensional phenotypic trait space $\R^N$. The function $x \mapsto - W(x)$ represents the {\it fitness} of the phenotype $x\in \R^N$ and models the individual reproductive success, and $t \mapsto -\overline W[u](t)$ stands for the mean fitness of the population at time $t$.

Our main goal in this paper is to understand, under appropriate assumptions, the long time behaviour of the Cauchy problem associated with \eqref{eq:rep-mut}. To do so, we prove two intermediate results on the underlying linear problem, which we believe are interesting in their own right. First,  concerning 
the linear eigenvalue problem 
\begin{equation}
\label{eq:eigenvalue}
-\sigma^2(J*u-u)+W(x)u=\lambda u \quad \text{ in } \R^N,
\end{equation}
we prove that there exists a unique nonnegative principal eigenfunction $\phi$ with total mass one, which is also called a Perron eigenfunction or {\it ground state} for such problems, see Theorem~\ref{thm:eigen}. Next, concerning the long time dynamics of the linear Cauchy problem 
\begin{equation}
\label{eq:parabolique}
\partial_{t}u = \sigma ^2(J*u-u)-W(x)u, \quad t>0,\; x\in \R^N,
\end{equation}
we show that it is well-posed in a family of relevant Banach spaces, and we prove that its long time dynamics is determined by the principal eigenfunction $\phi$, see Theorem~\ref{thm:asympto}. From this point on, we are in a position to establish our result on the nonlinear Cauchy problem~\eqref{eq:rep-mut}: we prove that any solution converges, as the time $t\to+\infty$, to a multiple of the ground state, see Theorem~\ref{thm:rep-mut}.

\medskip

Replicator-mutator models aim at describing Darwinian evolutionary processes,  whose fundamental principles are mutations and selection. Under the constraint of constant mass $\textstyle \int_\mathbb{R} u(t,x)\, dx=1$, the replicator dynamics   is given by
\begin{equation*}\label{nodiffusion}
\partial_{t}u = -\left(W(x) - {\overline W[u](t)}\right)u.
\end{equation*} 
As an attempt to take into account evolutionary phenomena, mutations are typically modelled by integral operators, thus yielding models such as \eqref{eq:rep-mut};  we refer to the influential work of M. Kimura \cite{Kimura1965}, as well as R. Lande \cite{Lan-75}, W.H. Fleming \cite{Fle-79} and R. B\"urger \cite{Bur-86, Bur-88}. In some situations, for instance those discussed in R. B\"urger \cite[Chapter VI, subsection 6.4]{Bur-00-book}, these integral operators can be approximated  by a local diffusion operator, as in
\begin{equation}
\label{eq:rep-mut-laplacien}
\partial_{t}u = \sigma ^2\Delta u -\left(W(x)-\overline W[u](t)\right)u.
\end{equation}
We refer to the work of N. Champagnat, R. Ferrière \& S. Méléard~\cite{Champagnat2008} or the recent paper by J.Y. Wakano, T. Funaki \& S. Yokoyama \cite{Wak-Fun-Yok-17} for a rigorous derivation of the replicator-mutator problem from individual based models.

The case of a linear fitness function, that is when $W(x) := -x$ (say $N=1$), was completely investigated in the works by M. Alfaro \& R. Carles~\cite{Alf-Car-14} (Laplacian case, Equation \eqref{eq:rep-mut-laplacien}) and R. Bürger~\cite{Bur-91} and M.E. Gil \& al. \cite{Gil-17} (mutation kernel case, Equation \eqref{eq:rep-mut}), whereas the quadratic case, that is $W(x) := -x^{2}$ (say $N=1$), was studied in M. Alfaro \& R. Carles~\cite{Alf-Car-17}  for Equation \eqref{eq:rep-mut-laplacien}. These two cases  share the property consisting in the fact that the function $W$ is unbounded from below, meaning that some phenotypes are infinitely well-adapted. These two cases  yield rich mathematical behaviours (acceleration, extinction) but,  unless introducing some {\it context-dependent} mutation kernels as performed by M.E. Gil \& al.  in \cite{Gil-Ham-Mar-Roq-18}, are not well-suited models for biological applications. 

On the other hand, the confining prototype case $W(x) := x^{2}$ (say $N=1$), considered in the pioneering work by M. Kimura~\cite{Kimura1965} and analyzed by R. Bürger~\cite{Bur-86} for Equation~\eqref{eq:rep-mut} and by M. Alfaro and R. Carles~\cite{Alf-Car-17} for Equation \eqref{eq:rep-mut-laplacien}, prevents extinction phenomena and leads to convergence to the underlying principal eigenfunction. In a related but different setting,  let us mention the recent work of F. Hamel \& al. \cite{Hamel2020} where a \lq\lq two-patches environment'' is considered. However, the case $W(x)=x^2$ does not suffice to  take into account more realistic cases for which fitness functions are defined by a linear combination of two components (e.g. birth and death rates), each maximized by different optimal values of the underlying trait, a typical case being $W(x) := x^4 - x^2$. In this setting, let us mention the recent work of M. Alfaro \& M. Veruete \cite{Alf-Ver-18} which provides a rigorous treatment of the Cauchy problem \eqref{eq:rep-mut-laplacien} when the fitness function $W$ is confining, and also raises the issue of  \emph{evolutionary branching},  consisting of the spontaneous splitting from uni-modal to multi-modal distribution of the trait. Our main goal is to extend these results to the mutation kernel model \eqref{eq:rep-mut}, revisiting the work of R. B\"urger~\cite{Bur-88} and refining it in the case when the convolution kernel $J$ is an even function (that is $J(x) = J(-x)$ for all $x\in {\Bbb R}^N$).

This requires to perform first a detailed analysis of the integral eigenproblem \eqref{eq:eigenvalue}. The local counterpart to \eqref{eq:eigenvalue}, namely
$$
-\sigma^2\Delta u + W(x)u = \lambda u,
$$
is  well understood, see for instance M. Reed \& B. Simon \cite{ReedSimonVol4}, L.A. Takhtajan \cite{Takhtajan}. In the quantum mechanics terminology, the principal eigenfunction $\phi$  is called the ground state and  corresponds to the bound-state having minimal energy, and the principal eigenvalue $\lambda_{1}$ is  characterized by a classical variational formulation.
In the non-local case \eqref{eq:eigenvalue}, the principal eigenvalue is not necessarily associated with a principal eigenfunction.
The lack of regularizing effect of the integral operator (notice that in dimension $N=1$, some regularization may be provided by an additional drift term, see \cite{Clo-Gab-20,Cov-Ham-20}) compared to the Laplace operator may result in situations where the ground state is a singular measure (containing atoms).
It happens when the fitness function $W$ is confining and has a cusp at its minimum, as noticed first by R. Bürger in~\cite{Bur-88} and then further investigated by R. Bürger \& I.M. Bomze in~\cite{Bur-Bom}.
More recently, this problem was revisited by J. Coville and co-authors in the case of a bounded domain \cite{Bon-Cov-Leg,Cov-10,Cov-13,Li-Cov-Wan}, see also Q. Griette \cite{Gri-19}.
In the study of these concentration phenomena (formation of Dirac masses), a first fundamental question is to identify sharp conditions on $J$ and $W$ that ensure the existence (or non-existence) of a principal eigenfunction.
Such conditions can be found in~\cite{Bur-88,Bur-Bom,Cov-10,Li-Cov-Wan}, and the first contribution of the present paper is to give a new milder criterion (see Assumption~\ref{ass:linkJW}) in the case of an even convolution kernel $J$.
We also recover, through a different method, the existing criteria in the non-even case, see Section~\ref{sec:non-sym}.
The second novel result is a quantification of the spectral gap in $L^2$, and accordingly of the rate of convergence to the equilibrium for Equation~\eqref{eq:rep-mut} in this space, under some stronger assumptions, see Theorem~\ref{thm:spectral-gap}.

Notice that in the works of R. Bürger or J. Coville, more general mutation kernels than convolution are usually considered, {\it i.e.} with $k(x,y)$ in place of $J(x-y)$.
To simplify the statement of the assumptions and limit technicalities in the proofs, we choose to treat only the convolution case, which is biologically relevant, see~\cite{Kimura1965}.
However, the proofs can be adapted to encompass more general kernels, replacing the evenness of $J$ by the symmetry of $k$ (namely $k(x,y)=k(y,x)$).


\section{Main results}\label{sec:main}

Before stating our main results, we give the notations and assumptions used throughout the paper.

For an almost everywhere positive measurable weight function $\rho$ we  denote by $L^2(\rho)$ the weighted Lebesgue space
$$L^2(\rho) := \left\{u : {\Bbb R}^N \dans {\Bbb R} \text{ such that } u \text{ is measurable and }\, \int_{{\Bbb R}^N}|u(x)|^2\rho(x) dx < \infty \right\}.$$
We shall denote by $\langle \cdot\,,\cdot\rangle$ the scalar product of $L^2({\Bbb R}^N)$ and by $\|\cdot\|_{L^2}$ its associated norm. 

For any choice of space $E$ among either of the Lebesgue spaces $L^{p}(\R^{N})$ with $1 \leq p <\infty$, the space of continuous functions converging to zero at infinity $C_{0}(\R^{N})$, or $\MM(\R^N)$ the space of bounded measures on ${\Bbb R}^N$, we denote by $\|\cdot\|_{E}$ the standard associated norm, namely the $L^p$ norm, the $L^\infty$ norm, or the total variation norm, respectively.

If $X,Y$ are two Banach spaces and $A : X \dans Y$ is a linear bounded operator we denote $\|A\|_{X \to Y}$ its norm in the space $\LL(X,Y)$.

Our precise assumptions on the mutation kernel and on the potential are the following.

\begin{assumption}[On the kernel $J$]\label{ass:kernel} The kernel $J:\R^{N}\dans \R$ satisfies
\begin{enumerate}
\item $J(x) \geq 0$  for almost all $x \in \R^N$.
\item There exists $r_0>0$ such that $J(x) > 0$ for almost all $x$ in the ball $B(0,r_{0})$.
\item $J\in L^{1}(\R ^{N})\cap L^{2}(\R ^{N})$ and $\displaystyle \int_{\R^N}J(x)dx = 1$.
\end{enumerate}
\end{assumption}

\begin{assumption}[Symmetry of the kernel $J$]\label{ass:J-symmetric} The kernel $J:\R^{N}\dans \R$ satisfies
$J(-x) = J(x)$ for almost all $x \in \R^N$.
\end{assumption}

\begin{assumption}[On the potential $W$]\label{ass:potential} The potential $W:\R^{N}\dans \R$ is a continuous function which satisfies
\begin{enumerate}
\item $W(x)\to +\infty$ as $|x| \to +\infty$ (confining assumption).
\end{enumerate}
Equation~\eqref{eq:rep-mut} is left unchanged by adding a constant to $W$, so we assume w.l.o.g. that
\begin{enumerate}[resume]
\item $W\geq 0$.
\end{enumerate}
\end{assumption}

In the sequel, for a function $f : \R^N \dans \R$ and a constant $\alpha \in \R$, the set of points $x \in \R^N$ such that $f(x) \geq \alpha$ is denoted by $[f \geq \alpha]$.

\begin{assumption}[Linking $J$, $\sigma$, and $W$]\label{ass:linkJW} There exist $\epsilon > 0$ and a Borel set $B\subset\R^N$ such that
\begin{equation}
\label{hyp:non-int}
\sigma^2\iint_{B_\epsilon\times B_\epsilon}\frac{J(x-y)}{W(x)W(y)}\,dxdy>\int_{B_\epsilon}\frac{1}{W(x)}\,dx
\end{equation}
where $B_\epsilon:=B\cap[W\geq\epsilon]$.
\end{assumption}

\medskip

We define the linear operator $(L,D(L))$ by setting
\begin{equation}\label{eq:Def-L}
Lu := - K*u + W(x) u, \quad\text{for }\, u \in D(L) := L^2(1+W), \qquad\text{where }\, K := \sigma^2J.
\end{equation}
It can be easily seen that $(L,D(L))$ is an unbounded operator on $L^2({\Bbb R}^N)$, and that when $J$ is an even function, that is when Assumption \ref{ass:J-symmetric} is satisfied, the operator $L$ is self-adjoint. Also, since Young's inequality yields $\|K*u\| _{L^2}\leq \|K\|_{L^1}\|u\|_{L^2} = \sigma^2\|u\|_{L^2}$ for all $u \in D(L) = L^2(1+W)$, we have
$$\langle Lu,u\rangle \geq - \sigma^2\|u\|_{L^2}^2,$$
so that when Assumption \ref{ass:J-symmetric} is satisfied, that is when $L$ is self-adjoint, the numerical range of $L$, and thus its spectrum, is contained in $[-\sigma^2,+\infty)$. In general neither $L$ nor its resolvent are compact, so in order to show that it has an eigenvalue, when $L$ is self-adjoint, we use the variational characterization of the bottom of its spectrum. Namely, setting 
\begin{equation}
\label{def:S}
S := \left\{u\in L^2(1+W)\text{ such that }  \int_{\R^N} u^2(x)dx = 1\right\},
\end{equation}
we define the energy functional
\begin{equation}
\label{def:energy}
\E(u):= \langle Lu,u\rangle = -\int_{\R ^N} (K*u)(x)u(x)dx + \int_{\R^N}W(x)u^2(x)dx, 
\end{equation}
and its infimum
\begin{equation}
\label{def:lambda-un}
\lambda_{1} := \inf_{u\in S} \E(u),
\end{equation}
and we show that under our assumptions $\lambda_{1}$ is achieved. Our first main result concerns the eigenvalue problem \eqref{eq:eigenvalue}, namely existence and uniqueness (up to a multiplicative constant) of a principal eigenfunction, or ground state.

\begin{theorem}[Principal eigenpair, symmetric case]\label{thm:eigen} 
Let Assumptions \ref{ass:kernel}, \ref{ass:J-symmetric}, \ref{ass:potential} and \ref{ass:linkJW} hold.
Then $\lambda_{1}$, defined by \eqref{def:lambda-un}, is achieved by a unique $\phi \in S \cap C_{0}(\R^N)$ such that $\phi > 0$ and
\begin{equation}\label{eq-vp}
-K*\phi + W(x)\phi = \lambda_{1} \phi\quad  \text{ in }\, \R ^N.
\end{equation}
Moreover,
\begin{enumerate}[label=\roman*)]
\item $-\sigma^2 < \lambda_{1}\leq -b_\epsilon$, where
\begin{equation}\label{eq:b-eps}
b_\epsilon:=\bigg(\sigma^2\iint_{B_\epsilon\times B_\epsilon}\frac{J(x-y)}{W(x)W(y)}\,dxdy-\int_{B_\epsilon}\frac{dx}{W(x)}\bigg)\bigg(\int_{B_\epsilon}\frac{dx}{W(x)}\bigg)^{-2}>0,
\end{equation}
\item $\displaystyle\forall x\in\R^N,\ 0 < \phi(x) \leq \frac{\|K\|_{L^2}}{W(x) - \lambda_{1}}$,
\item $\lambda_{1}$ is an eigenvalue of multiplicity one, that is if $\psi\in L^2(\R^N)$, or $\psi \in C_{0}(\R^N)$, is such that
\begin{equation*}\label{psi-ae}
-K*\psi + W(x)\psi = \lambda_{1} \psi \quad  \text{ a.e. in }\, \R^N,
\end{equation*}
then there is $\alpha \in \R$ such that $\psi = \alpha  \phi$.
\end{enumerate}
\end{theorem}

\medskip

 We now turn to the long time dynamics of the linear equation \eqref{eq:parabolique}.
Clearly, under the assumptions of Theorem \ref{thm:eigen} the pair $(\lambda_{*} ,\phi)$ with 
\begin{equation}\label{eq:Def-lambda-*}
\lambda_{*} := \lambda_1 + \sigma^2 > 0
\end{equation}
is solution to~\eqref{eq:eigenvalue}, and thus the function $u_{*}$ defined by
\begin{equation*}\label{eq:Def-u*}
u_{*}(t,x) := {\rm e}^{-\lambda_{*}t}\phi(x),
\end{equation*}
is solution to~\eqref{eq:parabolique}. In fact the large time behaviour of any nonnegative solution to \eqref{eq:parabolique} is given by this particular solution (recall that we denote by $E$ either of the spaces $L^p({\Bbb R}^N)$ for $1\leq p < \infty$, or $C_{0}({\Bbb R}^N)$ or $\MM({\Bbb R}^N)$).

\begin{theorem}[Long time dynamics, linear case]\label{thm:asympto}
Let Assumptions \ref{ass:kernel}, \ref{ass:J-symmetric}, \ref{ass:potential} and \ref{ass:linkJW} hold, and for any $u_{0}\in E$ denote by $\langle u_{0},\phi \rangle = \int_{\R^N}u_{0}(x)\phi(x)dx$ its weighted mass\footnote{As will be proved in Corollary \ref{villoise}, $\varphi \in L^1(\R^N)$ and thus, since $\varphi \in C_0(\R^N)$, $\varphi \in L^q(\R^N)$ for any $1\leq q \leq \infty$. In particular the weighted mass $\langle u_{0},\phi \rangle$ is finite as soon as $u_0\in E$.}.
Then there exist two constants $C > 0$ and $a > 0$ such that, for any $u_{0}\in E$, the solution $u=u(t,x)$ of \eqref{eq:parabolique} starting from $u_{0}=u_0(x)$ (see Section \ref{sec:asympto} for the precise definition) satisfies, for all $t > 0$,
$$
{\rm e}^{\lambda_{*}t} \big\|u(t,\cdot) - \langle u_{0},\phi \rangle u_{*}(t,\cdot)\big\|_{E} = \big\| {\rm e}^{\lambda_{*} t}u(t,\cdot)-\langle u_{0},\phi \rangle \phi\big\|_{E} \leq C\, {\rm e}^{-at}\left\|u_{0} - \langle u_{0},\phi \rangle\phi\right\|_{E} .
$$
\end{theorem}

\medskip

The next result deals with the nonlinear replicator mutator model \eqref{eq:rep-mut}.
This equation is concerned with the evolution of probability distributions, so that we only consider an initial data $u_{0}\in E$ which is nonnegative and such that $\langle u_{0},\1\rangle=1$, where $\1$ stands for the constant function equal to $1$ over $\R^N$.

\begin{theorem}[Long time dynamics, replicator-mutator model]\label{thm:rep-mut}
Let Assumptions \ref{ass:kernel}, \ref{ass:J-symmetric}, \ref{ass:potential} and \ref{ass:linkJW} hold.
Then there exists $a>0$ such that, for any $u_{0}\in E$ with $u_{0} \geq 0$ and $\langle u_{0},\1\rangle=1$, there is a constant $C(u_{0})>0$ for which the solution $u=u(t,x)$ of \eqref{eq:rep-mut} starting from $u_{0} = u_{0}(x)$ (see Section \ref{sec:rep-mut} for the precise definition) satisfies, for all $t > 0$,
$$
\left\| u(t,\cdot)-\frac{\phi}{\langle\phi,\1\rangle}\right\|_{E} \leq C(u_{0}) \, {\rm e}^{-at}.
$$
Moreover, in the case $E = L^1(\R^N)$ or $E = \MM(\R^N)$, there exists a constant $c_{0} > 0$ such that for any $u_{0} \in E$ with $u_{0} \geq 0$ and $\langle u_{0},\1\rangle=1$, we can take 
$$C(u_{0}) = c_{0}\left\|\frac{u_{0}}{\langle u_{0},\phi \rangle} - \phi\right \|_{E}.$$
\end{theorem}

\medskip

Our last main result provides a quantitative estimate of the rate of convergence $a$ appearing in Theorems \ref{thm:asympto} and \ref{thm:rep-mut}.
To do so we need to further assume that $W$ admits a unique global minimum, which we can choose w.l.o.g. to be the origin.

\begin{assumption}\label{ass:W-fort}
The potential $W:\R^N\dans\R$ is a continuous function which satisfies
\begin{enumerate}
\item $W(x)\to +\infty$ as $|x| \to +\infty$.
\item $W(0)=0$.
\item For all $x\neq0$, $W(x)>0$.
\end{enumerate}
\end{assumption}

Then we consider a convex, open and bounded set $\Omega\subset\R^N$ such that $0\in\Omega$, and we define
\[\eta:=\esssup_{2\Omega}K-\essinf_{2\Omega}K,\]
where $2\Omega=\{x\in\R^N:x/2\in\Omega\}$.
We also define the function $\Phi:[0,+\infty)\to\R\cup\{+\infty\}$ by
\begin{equation}\label{def:Phi}
\Phi(\xi):=\min\Big\{\sigma^2-\xi,\eta\meas(\Omega)+a_1\xi+a_2\sqrt{\xi}\Big\},
\end{equation}
where
\[a_1:=\Big(\essinf_{2\Omega}K\Big)\bigg(\int_{\Omega^c}\frac1W\bigg)\quad\text{and}\quad a_2:=2\sigma\sqrt{\sup_{x\in\R^N}\int_{\Omega^c}\frac{K(x-y)}{W(y)}dy},\]
and we denote by $\bar\Phi\in(0,+\infty]$ the largest value of $\Phi$
\[\bar\Phi:=\sup_{\xi\geq0}\Phi(\xi).\]

\begin{theorem}[Spectral gap and rate of convergence]\label{thm:spectral-gap}
Let Assumptions~\ref{ass:kernel}, \ref{ass:J-symmetric}, \ref{ass:linkJW} and \ref{ass:W-fort} hold.
Then, for any eigenvalue $\lambda$ of the operator $L$ in $L^2(\R^N)$, we have
\[\lambda\neq\lambda_1\quad\implies\quad \lambda\geq-\bar\Phi.\]
As a consequence, if the sets $\Omega$ and $B_\epsilon$ are such that
\[a_*:=b_\epsilon-\bar\Phi>0\]
then the convergences in Theorems~\ref{thm:asympto} and~\ref{thm:rep-mut} hold in $L^2(\R^N)$ for any $a<a_*$.
\end{theorem}

\medskip

Let us now comment on the above assumptions and results.

\begin{enumerate}
\item Assumption~\ref{ass:linkJW} is the crucial condition which ensures, when $J$ is even, that the ground state is a function and not a singular measure.
It relaxes the existing criteria, the mildest one in the literature being, to the best of our knowledge, the one proposed by F. Li, J. Coville \& X. Wang in~\cite{Li-Cov-Wan} (see Condition~(2.3) in Theorem~2.1).
In the case when $\meas[W=0]=0$, their condition can be written in the following form (see Appendix~\ref{app:ex1} for the proof of the equivalence).
\begin{assumption}\label{ass:Coville}
There exist $\epsilon>0$ and a Borel set $B\subset\R^N$ such that
\begin{equation*}\label{hyp:Coville}
\sigma^2\essinf_{x\in B_\epsilon}\int_{B_\epsilon}\frac{J(x-y)}{W(y)}\,dy>1
\end{equation*}
where we recall the notation $B_\epsilon=B\cap[W\geq\epsilon]$.
\end{assumption}
Note that this condition can also be proved to be sufficient from the paper \cite{Bur-88} by R. Bürger, see \cite[p.\,250, {\it Note added in proof.}]{Bur-Bom}. Clearly, Assumption~\ref{ass:Coville} implies  Assumption~\ref{ass:linkJW}, but the converse is false as shown by the following example.

\begin{example}\label{ex:1}
Consider the one dimensional space $\R^N=\R$ and
\[J(z)=\frac12\1_{[-1,1]}(z)\qquad\text{and}\qquad W(x)=\sqrt{|x|}.\]
Then Assumption~\ref{ass:Coville} is satisfied if and only if $\sigma^2>\frac 1{\sqrt 2}$, while Assumption~\ref{ass:linkJW} is verified as soon as $\sigma^2>\frac{4}{4+\pi}$.
We refer to Appendix~\ref{app:ex1} for a proof of this claim.
\end{example}

The question whether Assumption~\ref{ass:linkJW} is enough for ensuring the existence of a first eigenfunction without the evenness condition on $J$ is still an open question.

Finally, let us also mention that the condition
\[\sigma^2\esssup_{x\in \R^N}\int_{\R^N}\frac{J(x-y)}{W(y)}\,dy<1\]
guarantees that the ground state is a singular measure (with atoms), and consequently no principal eigenfunction exists, see~\cite{Bur-Bom}.

\item Since the work of S. Mischler and J. Scher~\cite{Mis-Sch} in 2016, quantifying the spectral gap of non-local and non-conservative linear equations is an active field of research, see~\cite{Ban-Clo-Gab-Mar,Can-Gab-Yol,Clo-Gab-21,Clo-Gab-20,Gab-Mar}.
To our knowledge, the result in Theorem~\ref{thm:spectral-gap} is the first quantified spectral gap result in the literature for Equation~\eqref{eq:parabolique}.
For some particular choices of coefficients, as the one in the following example, it provides an estimate of the spectral gap.

\begin{example}\label{ex:2}
Consider the one dimensional space $\R^N=\R$ and 
\[J(z)=\frac14\1_{[-2,2]}(z)\qquad\text{and}\qquad W(x)=|x|^m\ (m>1).\]
Then the lower bound $a_*$ on the spectral gap satisfies
\[a_*\geq\sigma^2\left(\frac14-\frac{\sigma^2}{m-1}-\sqrt 2\frac{\sigma}{\sqrt{m-1}}\right)\]
and is thus positive for $\sigma^2$ small enough or $m$ large enough.
We refer to Appendix~\ref{app:ex2} for a proof of this estimate.
\end{example}
\end{enumerate}

\medskip

Assumptions~\ref{ass:kernel} and \ref{ass:potential} are supposed to be verified throughout the paper, while the symmetry of $J$, that is Assumption~\ref{ass:J-symmetric}, is assumed only in Sections~\ref{sec:L}, \ref{sec:asympto} and \ref{sec:rep-mut}, Section~\ref{sec:non-sym} being devoted to the non-symmetric case.
The remainder of the paper is organized as follows. In Section \ref{sec:L} we gather some results concerning the operator $L$:
we give a strong maximum principle for this operator, we prove Theorem~\ref{thm:eigen} on the eigenvalue problem \eqref{eq:eigenvalue}, and we establish a functional inequality which yields the first part of Theorem~\ref{thm:spectral-gap} about the quantification of the spectral gap.
The long time behaviour of the linear problem \eqref{eq:parabolique}, see Theorem \ref{thm:asympto} and the second part of Theorem~\ref{thm:spectral-gap}, is studied in Section \ref{sec:asympto}. 
Theorem \ref{thm:rep-mut} on the replicator-mutator model \eqref{eq:rep-mut} is proved in Section \ref{sec:rep-mut}.
Finally, in Section~\ref{sec:non-sym}, we give a new proof of known results when $J$ is not assumed to be even.



\section{The operator $L$}\label{sec:L}

In this section, we investigate some remarkable properties of the operator $L$, defined in~\eqref{eq:Def-L}.

\subsection{A strong maximum principle}\label{sec:MaxPrinciple}

In the context of an elliptic second order equation such as
$$-\sum_{i,j=1}^N \partial_{i}(a_{ij}(x)\partial_{j}u) + (c(x) + \lambda) u = f \qquad \text{in }\, \R^N,$$
where the matrix $(a_{ij})_{1\leq i,j \leq N}$ is uniformly coercive and $a_{i,j},c \in L^\infty(\R^N)$, while for instance $f \in C_{0}(\R^N)\cap L^2(\R^N)$, it is well known that if $f \geq 0$ and $f \not\equiv 0$ then $u > 0$ in $\R^N$, provided $c(x) +\lambda > 0$ a.e. in $\R^N$. The nonlocal operator $L$ defined above in \eqref{eq:Def-L} satisfies an analogous strong maximum principle.

\begin{lemma}\label{lem:Max-principle} 
Assume that Assumptions \ref{ass:kernel} and \ref{ass:potential} are satisfied.
Let $\lambda > \sigma^2$ and $f \in L^p(\R^N)$ such that $ f \geq 0$ and $f \not\equiv 0$, with $1 \leq p \leq \infty$. Let $u \in L^p(\R^N)$ satisfy
$$L u + \lambda u = -K*u + (W(x)+ \lambda)u = f.$$
Then $u > 0$ a.e. on $\R^N$.
\end{lemma}

\begin{proof}
First, assuming that $1 \leq p < \infty$, we show that $u \geq 0$. Writing $u = u^+ - u^-$ with $u^+ := \max(u,0)$, $u^-:=\max(0,-u)$, we have 
$$K*u^- + (W(x)+ \lambda)u = f + K*u^+ \geq 0.$$
Multiplying this by $(u^-)^{p-1}\1_{[u < 0]}$ and integrating we get
$$\int_{\R^N}(K*u^-)(x)(u^-)^{p-1}\1_{[u < 0]}(x)dx - \int_{\R^N}(W(x) + \lambda)|u^-(x)|^{p}dx \geq 0.$$
However, using H\"older's inequality with $K*u^- \in L^p(\R^N)$ and $(u^-)^{p-1}\1_{[u < 0]} \in L^{p'}(\R^N)$ where we denote $p' := p/(p-1)$, we have 
\begin{eqnarray*}
\int_{\R^N}(K*u^-)(x)(u^-)^{p-1}\1_{[u < 0]}(x)dx & \leq \|K*u^-\|_{L^p}\|(u^-)^{p-1}\1_{[u < 0]}\|_{L^{p'}} \\
&\leq \|K\|_{L^1}\|u^-\|_{L^p}^p = \sigma^2\|u^-\|_{L^p}^p,
\end{eqnarray*}
and thus from the previous inequality we infer that
$$\int_{\R^N} (W(x) + \lambda - \sigma^2)|u^-(x)|^pdx \leq 0,$$
that is, since $W \geq 0$ and $\lambda > \sigma^2$, we have $u^- \equiv 0$, and thus $u = u^+ \geq 0$. 

Next assume that $p = \infty$, and denote by $m$ the essential infimum of $u$, that is
$$m := \essinf\limits_{x \in \R^N} u(x).$$
Since $K \geq 0$, we deduce that, for  a.e. $x \in \R^N$, we have
$$\int_{\R^N}K(x-y)u(y)dy \geq m \int_{\R^N}K(x-y)dy = m\,\sigma^2,$$
and thus 
$$(W(x) + \lambda)u(x) \geq f(x) + \int_{\R^N}K(x-y)u(y)dy \geq m\,\sigma^2.$$
Since $W \geq 0$ and $\lambda > \sigma^2$, this inequality implies that $m\geq 0$. Indeed, if $m < 0$, taking a sequence $(x_{n})_{n \geq 1}$ such that $u(x_{n}) \to m$
 as $n \to \infty$ and $m \leq u(x_{n+1}) \leq u(x_{n})$, then for $n$ large enough so that $u(x_{n}) < 0$ and
$$0 \leq u(x_{n}) - m \leq \frac{1}{2\lambda}(\lambda - \sigma^2)|m|,$$ 
we would have
$$0 > -\frac 12 (\lambda-\sigma^2)\vert m\vert> (\lambda - \sigma^2)m + \lambda (u(x_{n}) - m) \geq - W(x_{n})u(x_{n}) \geq 0,$$
which is a contradiction. Thus $m \geq 0$, that is $u \geq 0$ in $\R^N$.
\medskip

Now, in order to show that $u > 0$ almost everywhere, we introduce the continuous function $U:\R^N\to\R$ defined by
\[U(x)=\int_{B(0,r_0)}u(x-y)\,dy\]
where $r_0$ is defined in item 2 of Assumption \ref{ass:kernel}, and we consider the closed set $[U = 0]$.
If this set were not empty, then we may take $x_{0} \in \R^N$ such that $U(x_{0}) = 0$.
Since $u \geq 0$, we would have $u(x)=0$ for almost very $x\in B(x_0,r_0)$ and accordingly, by using Tonelli's theorem,
\begin{align*}
0 = \int_{B(x_0,r_0)}(W(x) + \lambda)u(x)\,dx &= \int_{B(x_0,r_0)}(f(x) + K*u(x))\,dx\\
& \geq \int_{\R^N} K(y)U(x_0•-y)\,dy \geq \int_{B(0,r_{0})} K(y)U(x_0-y)\,dy\geq0.
\end{align*}
Since $K>0$ a.e. on $B(0,r_0)$, we deduce that $U(x)=0$ for all $x\in B(x_0,r_0)$, that is whenever $x_{0} \in [U = 0]$ we have also $B(x_{0},r_{0}) \subset [U = 0]$. 
This means that the closed set $[U = 0]$ is also open, and $\R^N$ being a connected set, we infer that either the set $[U = 0] $ is empty, or it is all of $\R^N$.
The latter would imply $u\equiv0$, which is ruled out since $f\not\equiv 0$.
Thus $[U = 0] = \emptyset$, that is $U > 0$ on $\R^N$.
Since $K>0$ a.e. on $B(0,r_0)$, this necessarily implies that $\int_{B(0,r_0)}K(y)u(x-y)\,dy>0$ for all $x\in\R^N$,
and consequently
$$u(x)\geq\frac{1}{\lambda+W(x)}\int_{B(0,r_0)}K(y)u(x-y)\,dy>0,$$
and the strong maximum principle is proved.
\end{proof}

\subsection{The eigenvalue problem}\label{sec:eigen}

In this section we prove existence and uniqueness of a principal eigenfunction to \eqref{eq:eigenvalue}, that is we prove  Theorem \ref{thm:eigen}.
Before going further, recall the definitions of the function $K$ in \eqref{eq:Def-L}, the energy $\E(u)$ in \eqref{def:energy}, the eigenvalue candidate $\lambda_{1}$ in \eqref{def:lambda-un}, and that of the set $S$ in \eqref{def:S}. Observe that, for $u\in S$ we have
$$
\int_{\R^N} (K*u)u\leq \left\Vert K*u\right\Vert_{L^2}\left\Vert u\right\Vert_{L^2}\leq \left\Vert K\right\Vert_{L^1}\left\Vert u\right\Vert_{L^2}\left\Vert u\right\Vert_{L^2}=\left\Vert K\right\Vert_{L^{1}},
$$
and thus $\lambda_{1} \geq -\left\Vert K\right\Vert_{L^{1}} = -\sigma^2$.

We first take advantage of the condition \eqref{hyp:non-int} to prove the following.

\begin{lemma}\label{lem:energy-neg} There is $\phi_{*} \in S$ such that $\E(\phi_{*}) < 0$.
\end{lemma}

\begin{proof}
It is a direct consequence of Assumption~\ref{ass:linkJW}.
Indeed, due to the monotone convergence theorem, we can assume w.l.o.g. that the set $B$ in Assumption~\ref{ass:linkJW} is essentially bounded. So for any $C>0$ the function
$$
\phi_*(x) := C\frac{1}{W(x)}{\mathbf 1}_{B_{\epsilon}}(x)
$$
belongs to $L^2(1+W)$, and we can choose $C$ such that $\int_{\R^N}\phi_*^2(x)dx = 1$, {\it i.e.} $\varphi_*\in S$.
Then, recalling the definition of $b_\epsilon>0$ in~\eqref{eq:b-eps}, we have $
\E(\phi_*) = -b_\epsilon<0$. 
\end{proof}

The following is a sort of compactness result, or rather a weak sequential continuity, concerning the quadratic mapping 
$$u \mapsto \int_{\R^N}(K*u)(x)u(x)dx.$$

\begin{lemma}\label{lem:convergence} If $(u_n)_{n\geq0}\subset S$ verifies
\begin{equation}\label{cv-faible}
u_{n} \weak u \quad \text{in }\, L^2(1+W),
\end{equation}
then
$$
\int_{\R^{N}}(K*u_{n})u_{n} \to \int_{\R^{N}} (K*u)u.
$$
\end{lemma}

\begin{proof} Since $(u_{n})_{n}$ is weakly convergent in $L^2(1+W)$, we can set
$$M := \sup_{n \geq 1}\|u_{n}\|_{L^2(1+W)} < \infty.$$
Due to the Cauchy-Schwarz inequality  we have
\begin{equation}
\label{est-pointwise}
\vert (K* u_{n})(x)\vert\leq \Vert K\Vert_{L^2}\Vert  u_{n}\Vert_{L^2}=\Vert K\Vert_{L^2}.
\end{equation}
Next note that since $(u_{n})_{n}$ converges weakly to $u$ in $L^2(1+W)$, we have also $u_{n} \weak u$ in $L^{2}$. Now, for a given $x\in \R ^{N}$, we have $K(x - \cdot) \in L^2(\R^N)$ and therefore
$$
(K* u_{n})(x) = \int_{\R^N} K(x-y) u_{n}(y)dy \to \int_{\R^N} K(x-y) u (y)dy = (K* u)(x).
$$
From this and \eqref{est-pointwise}, using the Lebesgue dominated convergence theorem we deduce that 
\begin{equation}
\label{L2loc}
K* u_{n} \to K* u \quad \text{ in } L^{2}_{\rm loc}(\R^{N}).
\end{equation}

Now, let $\epsilon > 0$ be given. We may choose $R > 0$ large enough so that $(1+W(x))^{-1/2}\leq \epsilon$ when $|x| > R$. Then
\begin{eqnarray*}
\left\vert \int_{\vert x\vert >R} (K* u_{n}) u_{n}\right\vert&\leq &\ep \int_{\vert x\vert >R}(K*\vert  u_{n}\vert) (1+W)^{\frac 12}\vert  u_{n}\vert\\
&\leq &\ep\, \Vert K*\vert  u_{n}\vert \,\Vert_{L^2}\;\Vert (1+W)^{\frac 12}\vert  u_n\vert\, \Vert_{L^2}\\
&\leq &\ep\, \Vert K\Vert_{L^{1}}\;\Vert  u_{n} \Vert_{L^2}\;\Vert (1+W)^{\frac 12} u_{n} \Vert_{L^2}\\
&\leq & \epsilon M^2\|K\|_{L^1},
\end{eqnarray*}
and it is clear that the same estimate holds for $\left\vert \int_{\vert x\vert >R} (K* u) u\right\vert$. As a result
\begin{multline*}
\left\vert \int_{\R^{N}} (K* u_{n}) u_{n}-\int_{\R^{N}} (K* u) u
\right\vert\\
\leq \left\vert\int_{\vert x\vert \leq R} (K*( u_{n}- u)) u_{n}\right\vert+\left\vert\int_{\vert x\vert \leq R} (K* u)( u_{n}- u)\right\vert + 2\epsilon M^2\|K\|_{L^1}.
\end{multline*}
As $n\to+\infty$, the first and second terms in the right hand side tend to zero due to \eqref{L2loc} and \eqref{cv-faible} respectively. This concludes the proof of the lemma.
\end{proof}

We are now in a position to prove our main result concerning the eigenvalue problem~\eqref{eq:eigenvalue}, namely the existence of a unique (up to normalization) principal eigenfunction.

\begin{proof}[Proof of Theorem \ref{thm:eigen}] We consider a sequence $\phi_{n} \in S$ such that
$$
\lambda_{1}\leq \E(\phi_{n})\leq \lambda_{1} + \frac 1n.
$$
Since $\E(|u|)\leq \E(u)$, up to replacing $\phi_{n}$ by $|\phi_{n}|$ we can assume $\phi_{n}\geq 0$. We have also 
\begin{eqnarray*}
\int_{\R^N}W(x)\phi_{n}^2(x)dx  &=& \displaystyle \E(\phi_{n}) 
+ \int_{\R^N}(K*\phi_{n})(x)\phi_{n}(x)dx \\
& \leq &\lambda_{1} + 1 + \|K\|_{L^1}\|\phi_{n}\|_{L^2}^2 = \lambda_{1} + 1 + \sigma^2,
\end{eqnarray*}
so that $(\phi_{n})_{n}$ is bounded in $L^2(1+W)$, and thus there exists $\phi \in L^2(1+W)$ and a subsequence, denoted again by $(\phi_{n})_{n}$, such that
\begin{equation}\label{cv-faible-bis}
\phi_{n} \weak \phi \quad\text{and}\quad \sqrt W\phi_{n} \weak \sqrt W\phi\quad \text{ in } L^2.
\end{equation}
Next, using Lemma \ref{lem:convergence}, we have
\begin{equation}\label{ineg}
0\leq \int_{\R^N}  W(x)\phi_{n}^2(x)dx = \E(\phi_{n}) + \int_{\R^N} (K*\phi_{n})\phi_{n} \to \lambda_{1}+\int_{\R^N} (K*\phi)\phi.
\end{equation}
Since by Lemma \ref{lem:energy-neg} we have $\lambda_{1} < 0$, the above inequality implies that $\phi\not\equiv 0$ and $\phi \geq 0$.
On the other hand, thanks to the weak convergences given in \eqref{cv-faible-bis} we have
$$
\int_{\R^N} W(x)\phi^2(x)dx \leq \liminf_{n\to +\infty}\int_{\R^N}W(x)\phi_{n}^2(x) dx = \lambda_{1} + \int_{\R^N}(K*\phi)\phi ,
$$
and also
$$\int_{\R^N}\phi^2(x)dx \leq \liminf_{n\to +\infty} \int_{\R^N}\phi_{n}^2(x)dx = 1.$$
Thus $\E(\phi)\leq \lambda_{1}$ and $\theta^2 := \int_{\R^N}\phi^2(x)dx \leq 1$. Since $\phi\not\equiv 0$, setting $
{\widetilde \phi} := \theta^{-1}\phi$ we have ${\widetilde \phi} \in S$
and
$$
\lambda_{1} \leq \E(\widetilde \phi) = \theta^{-2}\E(\phi)\leq \theta^{-2}\lambda_{1} \leq \lambda_{1},
$$
where in the last inequality we use the fact that $\lambda_{1} < 0$. Clearly this implies that $\theta^2 = 1$ and thus ${\widetilde \phi} = \phi \in S$: this means that $\|\phi_{n}\|_{L^2} \to \|\phi\|_{L^2}$ while $\phi_{n} \weak \phi$ in $L^2(\R^N)$, yielding that the convergence of $(\phi_{n})_{n}$ to $\phi$ is strong. The same above inequalities imply also that $\E(\phi) = \lambda_{1}$, while from \eqref{ineg} we infer that
$$
\int_{\R^N}W(x)\phi_{n}^2(x)dx \to \int_{\R^N}W(x)\phi^2(x)dx,
$$
that is $\|\phi_{n}\|_{L^2(1+W)} \to \|\phi\|_{L^2(1+W)}$, again yielding that $\phi_{n} \to \phi$ in $L^2(1+W)$. 

Finally we have $\phi \in S$ and $\E (\phi) = \lambda_{1}$. Since here Assumption \ref{ass:J-symmetric} is satisfied, $L$ is self-adjoint and therefore there exists a Lagrange multiplier $\lambda\in \R$ such that
$$L\phi = -K*\phi+W(x)\phi = \lambda \phi \quad  \text{ a.e. in }\, \R ^N,$$
and, obviously upon multiplying this equation by $\phi$, one sees that $\lambda = \lambda_{1}$. We thus have 
\begin{equation}\label{egalite}
(W(x)-\lambda_{1})\phi(x) = (K*\phi)(x) \quad\text{that is}\quad
\phi(x) = \frac{(K*\phi)(x)}{W(x) - \lambda_{1}}.
\end{equation}
Since $K\in L^2(\R^N)$ and $\phi \in L^2(\R^N)$, we know that $K*\phi \in C_{0}(\R^N)$. As a result, from \eqref{egalite}, the continuity of $W$ and $W(x)-\lambda_{1} \geq -\lambda_{1} > 0$, we also have  $\phi\in C_{0}(\R^{N})$, and \eqref{eq-vp} holds.

Now, in order to see that $\phi > 0$, recalling that $\phi \in C_{0}(\R^N) \cap L^2(\R^N)$ and  $\phi \geq 0$ satisfies
$$-K*\phi + (W(x) - \lambda_{1} + 1)\phi = \phi,$$
by the strong maximum principle, see Lemma \ref{lem:Max-principle}, we have $\phi > 0$. Also, using \eqref{est-pointwise} and \eqref{egalite}, we deduce the pointwise estimate
\begin{equation*}
0 < \phi(x)\leq \frac{\Vert K\Vert_{L^2}}{W(x)-\lambda_{1}}, \quad \forall x \in \R^N.
\end{equation*}

Once we know that $\phi > 0$ on $\R^N$, we can show  $\lambda_{1} > -\sigma^2$. Indeed, multiplying equality
$$(-\lambda_{1} + W(x))\phi = K*\phi$$
by $\phi$ and integrating, since $\phi > 0$ and $W \not\equiv 0$ is nonnegative, we get
$$-\lambda_{1} <  \int_{\R^N}(-\lambda_{1} + W(x))\phi^2(x)dx = \int_{\R^N}(K*\phi)(x)\phi(x)dx \leq \|K\|_{L^1} = \sigma^2,$$
where we have used H\"older's inequality on the right hand side together with Young's inequality $\|K*\phi\|_{L^2} \leq \|K\|_{L^1}\|\phi\|_{L^2}$, and the fact that $\|\phi\|_{L^2} = 1$, while $\|K\|_{L^1} = \sigma^2$.

It remains to prove the uniqueness of $\phi$, or in other terms the fact that the eigenspace corresponding to $\lambda_{1}$ has dimension one: that is if $\psi \in L^{2}(\R^{N})$ satisfies 
$$\psi\not\equiv 0, \qquad -K*\psi + W(x)\psi = \lambda_{1} \psi\qquad \text{in }\, \R^N, $$
then for a constant $\alpha \in \R$ we have $\alpha\psi = \phi$. Arguing as above, we conclude first that $\psi \in C_{0}(\R^{N})$ and, without loss of generality we may assume that there exists $x^{*} \in \R^N$ such that $\psi(x^{*}) > 0$, at the cost of replacing $\psi$ by $-\psi$, if necessary. Next, let $R > 0$ be large enough so that $W(x)-\lambda_{1} - \sigma^2 > 0$ for $|x| > R$, where we recall that $\sigma^2 = \|K\|_{L^1}$ (this is possible thanks to the fact that $W$ is confining). Since $\phi > 0$, we can choose $\epsilon > 0$ small enough so that $u_{\epsilon} := \phi - \epsilon \psi >0$ on $B(0,R)$. Let us now prove that
\begin{equation}
\label{eq:claim-partout}
u_{\epsilon} \geq 0 \quad \text{ on the whole of }\, \R^{N}.
\end{equation} 
If this were not true, then using the fact that $u_{\epsilon} \in C_{0}(\R^{N})$, we infer that $u_{\epsilon}$ achieves its global negative minimum at some $x_{0} \in \R^N$, and we necessarily have $|x_{0}| > R$. Since on the one hand
$$K*u_{\epsilon}(x_{0}) - \sigma^2 u_\ep(x_{0}) = \int_{\R^N} K(y)\left(u_{\epsilon}(x_{0} - y) - u_{\epsilon}(x_{0})\right) dy \geq 0,$$
and on the other hand, using the linear equations satisfied by $\phi$ and $\psi$, we have
$$
K*u_{\epsilon}(x_{0}) - \sigma^2 u_{\epsilon}(x_{0}) = (W(x_{0}) - \lambda_{1} - \sigma^2) u_{\epsilon}(x_{0}) < 0,
$$
we have a contradiction, which implies that \eqref{eq:claim-partout} holds.

Now, since $\psi(x^{*}) > 0$, we point out that if $u_{\epsilon}(x) \geq 0$ in $\R^N$, in particular $u_{\epsilon}(x_{*}) \geq 0$ and thus $\epsilon \leq \phi(x_{*})/\psi(x_{*})$. Hence we can define the real number
\begin{equation}\label{eq:Def-alpha-A}
\alpha := \sup A, \quad \text{where}\quad A := \left\{\epsilon > 0: u_{\epsilon} := \phi - \epsilon \psi \geq 0 \,\text{ on }\, \R^{N} \right\} ,
\end{equation}
and we know that $0 < \alpha \leq \phi(x_{*})/\psi(x_{*})$.
In particular we infer that if we set $u_{\alpha} := \phi - \alpha \psi$ then $u_{\alpha} \geq 0$ and satisfies
$$u_{\alpha} \in C_{0}(\R^N), \qquad
-K*u_{\alpha} + (W(x) - \lambda_{1} + 1)u_{\alpha} = u_{\alpha} \geq 0.$$
However, if we had $u_{\alpha} \not\equiv 0$, thanks to Lemma \ref{lem:Max-principle} we would have $u_{\alpha} > 0$ in $\R^N$ and there would exist $\epsilon_{0} > 0$ small enough such that $u_{\alpha + \epsilon_{0}} := u_{\alpha} - \epsilon_{0} \psi > 0$ on the ball $B(0,R)$. Proceeding as in the proof of \eqref{eq:claim-partout}, we would deduce that $u_{\alpha + \epsilon_{0}} \geq 0$ in $\R^N$ and thus $\alpha + \epsilon_{0} \in A$, the set defined in \eqref{eq:Def-alpha-A}, contradicting the definition of $\alpha$. Therefore we must have $u_{\alpha} \equiv 0$, that is $\phi = \alpha \psi$.
\end{proof}


\subsection{A quantified spectral gap result in $L^2(\R^N)$}

In this section we suppose that Assumption~\ref{ass:W-fort} is verified, and we consider a convex, open and bounded set $\Omega\subset\R^N$ that contains the origin.
Then we have the following results.

\begin{lemma}\label{lm:fct-ineq}
For all $u\in D(L)=L^2(1+W)$ such that $\int_{\R^N}u=0$ and $u\not\equiv0$ we have
\[\frac{\langle-Lu,u\rangle}{\|u\|_{L^2}^2}\leq \Phi\bigg(\frac{\int_{\Omega^c}Wu^2}{\|u\|_{L^2}^2}\bigg),\]
where the function $\Phi$ is defined in~\eqref{def:Phi}.
\end{lemma}
\begin{proof}
First, we clearly have, for all $u\in L^2(1+W)$,
\begin{equation}
\label{qqch0}
\langle-Lu,u\rangle=\langle K*u,u\rangle-\langle Wu,u\rangle\leq\|K\|_{L^1}\|u\|_{L^2}^2-\int_{\Omega^c}Wu^2=\sigma^2 \|u\|_{L^2}^2-\int_{\Omega^c}Wu^2.
\end{equation}
The second part of the minimum defining the function $\Phi$ in \eqref{def:Phi} deserves more attention and is valid only under the condition $\int_{\R^N}u=0$.
Due to the non-negativity of $W$ we have
\begin{equation}\label{qqch}
\langle-Lu,u\rangle\leq\langle K*u,u\rangle=\iint_{\Omega\times\Omega}K(x-y)u(x)u(y)\,dxdy+\iint_{(\Omega\times\Omega)^c}K(x-y)u(x)u(y)\,dxdy.
\end{equation}
We start by estimating the first term.
Using that $K$ is symmetric we have
\begin{align*}
\iint_{\Omega\times\Omega}&K(x-y)u(x)u(y)\,dxdy=\iint_{\Omega\times\Omega}K(x-y)u^+(x)u^+(y)\,dxdy\\
&\qquad+\iint_{\Omega\times\Omega}K(x-y)u^-(x)u^-(y)\,dxdy-2\iint_{\Omega\times\Omega}K(x-y)u^+(x)u^-(y)\,dxdy\\
&\leq\big(\esssup_{2\Omega}K\big)\bigg(\Big(\int_\Omega u^+\Big)^2+\Big(\int_\Omega u^-\Big)^2\bigg)-2\big(\essinf_{2\Omega}K\big)\Big(\int_\Omega u^+\Big)\Big(\int_\Omega u^-\Big)\\
&\leq\big(\essinf_{2\Omega}K\big)\bigg(\int_\Omega u^+-\int_\Omega u^-\bigg)^2+\eta\bigg(\Big(\int_\Omega u^+\Big)^2+\Big(\int_\Omega u^-\Big)^2\bigg)\\
&\leq\big(\essinf_{2\Omega}K\big)\bigg(\int_\Omega u\bigg)^2+\eta\meas(\Omega)\,\|u\|_{L^2}^2\,.
\end{align*}
Since $\int_{\R^N}u=0$, using the Cauchy-Schwarz inequality,
\[\bigg(\int_\Omega u\bigg)^2=\bigg(\int_{\Omega^c} u\bigg)^2\leq\bigg(\int_{\Omega^c}\frac1W\bigg)\bigg(\int_{\Omega^c}Wu^2\bigg).\]
As a result
\begin{equation}
\label{qqch2}
\iint_{\Omega\times\Omega}K(x-y)u(x)u(y)\,dxdy\leq \eta\meas(\Omega)\,\|u\|_{L^2}^2+(\essinf_{2\Omega}K\big)\bigg(\int_{\Omega^c}\frac1W\bigg)\bigg(\int_{\Omega^c}Wu^2\bigg).
\end{equation}
For the second term, using again the symmetry of $K$, we have
\begin{align*}
\iint_{(\Omega\times\Omega)^c}K(x-y)u(x)u(y)\,dxdy&\leq \iint_{(\Omega\times\Omega)^c}K(x-y)\vert u(x)u(y)\vert \,dxdy\\
&=\iint_{\Omega^c\times\R^N}K(x-y)\vert u(x)u(y)\vert \,dxdy\\
&\qquad+\iint_{\Omega\times\Omega^c}K(x-y)\vert u(x)u(y)\vert \,dxdy\\
&\leq2\int_{\Omega^c}(K*\vert u\vert)(x)\vert u (x)\vert \,dx\\
&\leq2\sqrt{\int_{\Omega^c}\frac{((K*\vert u\vert )(x))^2}{W(x)}dx}\,\sqrt{\int_{\Omega^c}W(x)u^2(x)\,dx}\,.
\end{align*}
Since
\[((K*\vert u\vert )(x))^2\leq\|K\|_{L^1}\int_{\R^N}K(x-y)u^2(y)\,dy,\]
we get
\begin{align}
\iint_{(\Omega\times\Omega)^c}K(x-y)u(x)u(y)\,dxdy&\leq 2\sqrt{\|K\|_{L^1}\int_{\R^N}\Big(\int_{\Omega^c}\frac{K(x-y)}{W(x)}dx\Big)u^2(y)dy}\,\sqrt{\int_{\Omega^c}Wu^2}\nonumber \\
&\leq2\sigma\|u\|_{L^2}\,\sqrt{\sup_{y\in\R^N}\int_{\Omega^c}\frac{K(x-y)}{W(x)}dx}\,\sqrt{\int_{\Omega^c}Wu^2}.\label{qqch3}
\end{align}
In view of \eqref{def:Phi}, it now suffices to combine \eqref{qqch0}, \eqref{qqch}, \eqref{qqch2} and \eqref{qqch3} to prove the result.
\end{proof}

\begin{corollary}\label{cor:spectral-gap}
If $\lambda\in\R$ is an eigenvalue of $L$ such that $\lambda\neq\lambda_1$, then $\lambda\geq -\bar\Phi:=-\sup_{[0,\infty)}\Phi$. 
\end{corollary}

\begin{proof}
Let $\lambda\in\R$ (recall that since $L$ is self-adjoint it has a real spectrum) and $\psi\in D(L)$ such that $L\psi=\lambda\psi$ with $\lambda\neq\lambda_1$.
Then necessarily $\langle\psi,\varphi\rangle=0$, where $\varphi>0$ is the principal eigenfunction, so that $\psi$ cannot be of constant sign.
If $\int_{\R^N}\psi=0$ then Lemma~\ref{lm:fct-ineq} applied to $u=\psi$ immediately ensures that $-\lambda\leq\bar\Phi$.
If $\int_{\R^N}\psi\neq0$, there exists $\alpha\in\R$ such that $\tilde \psi:=\psi+\alpha\varphi$ verifies $\int_{\R^N}\tilde\psi=0$
and Lemma~\ref{lm:fct-ineq} applied to $u=\tilde\psi$ yields, using that $\lambda\geq\lambda_1$, $\|\varphi\|_{L^2}=1$ and $\langle\psi,\varphi\rangle=0$,
\[-\lambda(\|\psi\|_{L^2}^2+\alpha^2)\leq-\lambda\|\psi\|_{L^2}^2-\lambda_1\alpha^2=\langle-L\tilde\psi,\tilde\psi\rangle\leq\bar\Phi\|\tilde\psi\|_{L^2}^2=\bar\Phi(\|\psi\|_{L^2}^2+\alpha^2),\]
which concludes the proof.
\end{proof}

This result provides a quantified estimate of the distance between $\lambda_1$ and the other eigenvalues of $L$ provided that an upper bound smaller than $-\bar\Phi$ is known for $\lambda_1$.

\section{Long time asymptotics of the linear problem}
\label{sec:asympto}

This section is devoted to the linear evolution equation~\eqref{eq:parabolique}. Recalling that the eigenpair $(\lambda_{*},\phi)$ satisfies \eqref{eq:eigenvalue}, with $\lambda_{*} := \lambda_{1} + \sigma^2 > 0$, defined in \eqref{eq:Def-lambda-*}, and $K := \sigma^2 J$,
we readily observe that the solutions $u=u(t,x)$ to \eqref{eq:parabolique} are related to the solutions $v=v(t,x)$ of the abstract Cauchy problem
\begin{equation}\label{eq:evolution-A}
\left\{\begin{array}{l}
\displaystyle\frac {dv(t)}{dt} = \A v(t)  \qquad\text{for }\, t > 0,
\vspace{2mm}\\
v(0) = u_{0},
\end{array}\right.
\end{equation}
where the operator $(\A,D(\A))$ and the function $v(t)$ are defined by ($L$ being as in \eqref{eq:Def-L}):
\begin{equation}\label{eq:Def-A-u-v}
D(\A) := D(L), \qquad
\A v := -Lv + \lambda_{1}v, \qquad\text{and}\quad
u(t,\cdot) = {\rm e}^{-\lambda_{*}t} v(t,\cdot).
\end{equation}

Recall that a function $v(t)$ is called a {\it classical solution} of Equation~\eqref{eq:evolution-A} if it lies in $D(A)$, is continuously differentiable, and~\eqref{eq:evolution-A} holds.
It is called a {\it mild solution} if $\int_0^tv(s)ds\in D(A)$ for all $t\geq0$ and
\[v(t)=u_0+A\int_0^tv(s)\,ds.\]

From Theorem \ref{thm:eigen}, we know that $\phi$ is the unique positive steady state (up to normalization) of Equation~\eqref{eq:evolution-A} in $C_{0}(\R^{N})$ and in $L^2(\R^{N})$, and thus proving Theorem \ref{thm:asympto} is tantamount to showing that positive solutions of \eqref{eq:evolution-A} converge to (a multiple of) this stationary solution. That is why in this section we shall work with the modified equation \eqref{eq:evolution-A}.

\medskip

To analyse the long time behaviour of Equation~\eqref{eq:evolution-A}, we take advantage of the theory of strongly continuous semigroups, also called $C_{0}$-semigroups, of positive linear operators.
There is a large literature on this field, but the standard references K.~Yosida~\cite{YosidaK}, W.~Arendt \& al.~\cite{Nagel86} and K.J.~Engel \& R.~Nagel~\cite{EN} will be enough here.
Recall that we study Equation~\eqref{eq:evolution-A} in one of the following Banach lattices: $E = L^p(\R^N)$ with $1\leq p<\infty$, or $E = C_{0}(\R^N)$, or $E = \MM(\R^N) = (C_{0}(\R^N))'$, equipped with the norm $\Vert \cdot \Vert_E$ denoting the $L^p$ norm, or the $L^\infty$ norm, or the total variation norm, respectively.

\subsection{Analysis in the space $L^2(\R^N)$}

To begin with, let us study Equation~\eqref{eq:evolution-A} in the Lebesgue space $L^2(\R^N)$. Since $(L,D(L))$ is a self-adjoint operator acting in $L^2(\R^N)$, and since for $v \in D(L)$ we have $\langle Lv,v\rangle \geq \lambda_{1}\|v\|^2$, by the very definition of $\A$ by \eqref{eq:Def-A-u-v} we conclude that $(\A,D(\A))$ is self-adjoint and $\langle \A v,v\rangle \leq 0$, that is $\A$ is an $m$-dissipative operator.
Also since $A\phi = 0$, this means that zero is the principal eigenvalue of the operator $(\A,D(\A))$, and that its spectrum $\sigma(A)$ is contained in $(-\infty,0]$.

Thus by the Hille-Yosida theorem (see for instance K.~Yosida \cite[Chapter IX, Section 8]{YosidaK}) $\A$ generates a $C_{0}$-semigroup of contractions which we shall denote by $(T_{t})_{t\geq 0}$, or sometimes by $T_{t} = \exp(t\A) = {\rm e}^{t\A}$. Moreover, since $(\A,D(\A))$ is self-adjoint, $T_{t}$ is also self-adjoint on $L^2(\R^N)$ and the semigroup $(T_{t})_{t \geq 0}$ is analytic, that is for any $u_{0} \in L^2(\R^N)$ we have $T_{t}u_{0} \in D(\A)$ for $t > 0$. In particular for any $u_{0} \in L^2(\R^N)$ the function $v(t) := T_{t}u_{0}$ is the unique solution of equation \eqref{eq:evolution-A} in the classical sense on the interval $(0,\infty)$.

Note that since $(\A,D(\A))$ is $m$-dissipative, we have $\|T_{t}\|_{L^2(\R^N) \to L^2(\R^N)} \leq 1$, but since $A\phi = 0$ we have $T_{t}\phi = \phi$ for all $t \geq 0$ and thus
$$\|T_{t}\|_{L^2(\R^N) \to L^2(\R^N)} = 1.$$
This implies that the growth bound of the semigroup $(T_{t})_{t\geq0}$, that is the real number
$$\omega_{0}(\A) := \inf\left\{w \in\R: \exists M >0,\; \forall t \geq0,\; \|T_{t}\|_{L^2(\R^N) \to L^2(\R^N)}\leq M{\rm e}^{w t}\right\},$$
is equal to zero.
Besides, since $(T_{t})_{t\geq0}$ is analytic, the spectral bound of the operator $A$, that is
$$\s(\A) := \sup\{\Re\lambda : \lambda\in\sigma(\A)\}$$
is equal to the growth bound of the semigroup $(T_{t})_{t\geq0}$ generated by $A$
(see for instance K.J.~Engel \& R. Nagel~\cite[Corollary IV.3.12]{EN}).
We conclude that
\begin{equation*}\label{eq:s-equal-omega0}
\s(\A) = \omega_{0}(\A) = 0.
\end{equation*}

As it is customary in the study of large time behaviour of solutions to linear evolution equations, we wish to show that there is a {\it gap} in the spectrum of $\A$, in the sense that there exists a number $a > 0$ such that
$$\sigma(\A) \setminus \{0\} \subset (-\infty,-a).$$
Once this is shown, then it is not difficult to see that, if $v(t)$ is the solution of \eqref{eq:evolution-A} its orthogonal projection on the space $(\R\phi)^\perp$ converges to zero at least as fast as ${\rm e}^{-at}$.
Indeed the restriction $A_{|(\R\varphi)^\perp}$ of $A$ to the invariant subspace $(\R\phi)^\perp$ verifies in this case
$$\omega_0(A_{|(\R\varphi)^\perp})=\s(A_{|(\R\varphi)^\perp})<-a.$$

For proving the existence of a spectral gap, we use the notion of {\it essential growth bound},
which is defined similarly as the growth bound.
First, we define the essential norm of a bounded linear operator $T$ in a Banach space $E$ by
\[{\|T\|}_\ess:=\inf\big\{{\|T-\mathcal K\|}_{E\to E}\,:\, \mathcal K:E\dans E\ \text{is compact}\big\}.\]
Then we define the essential growth bound of a semigroup $(T_t)_{t\geq0}$ in $E$ by
$$\omega_\ess(\A) := \inf\left\{w \in\R: \exists M >0,\; \forall t \geq0,\; {\|T_{t}\|}_\ess\leq M{\rm e}^{w t}\right\}.$$
Clearly, $\omega_\ess(A)\leq\omega_0(A)$, and a semigroup $(T_t)_{t\geq0}$ is said to be quasi-compact if $\omega_\ess(A)<0$.
The usefulness of the essential growth bound lies in the following result (see for instance K.J.~Engel \& R. Nagel~\cite[Corollary IV.2.11]{EN}):
\begin{center}\begin{minipage}{.9\linewidth}
For every $w>\omega_\ess(A)$ the set $\sigma(A)\cap\{\lambda\in\C: \Re\lambda\geq w\}$ is composed of a finite number of eigenvalues with finite algebraic multiplicity.
\end{minipage}\end{center}
As a consequence, for our self-adjoint semigroup in $L^2(\R^N)$, if we can prove that $\omega_\ess(A)<0$, that is $(T_t)_{t\geq0}$ is quasi-compact, then we immediately get the existence of $a\in(0,-\omega_\ess(A))$ such that $\sigma(A)\setminus\{0\}\subset(-\infty,-a)$.

\medskip

In order to prove that $\omega_\ess(A)<0$, we split the operator $\A$ defined in~\eqref{eq:Def-A-u-v} as the sum of a local unbounded operator, namely
$$\A_{0} u:=\lambda_{1} u - W(x) u,\qquad D(\A_{0}) = D(L) = \{u\in E: (1+W) u \in E\},$$
and a nonlocal bounded perturbation, given by
$$\BB u := K*u = \sigma^2 J*u,$$
where we have $\|\BB\|_{L^2(\R^N) \to L^2(\R^N)}\leq\sigma^2$. 

It is straightforward to see that the operator $(\A_{0},D(\A_{0}))$ generates a $C_{0}$-semigroup of contractions which we shall denote by $(S_{t})_{t \geq 0}$, and as a matter of fact it can be written explicitly, not only in the space $L^2(\R^N)$ but in any of the spaces $E$ defined above.

\begin{lemma}\label{lem:St-analytique}
The unbounded operator $\big(\A_{0},D(\A_{0})\big)$ generates a positive $C_{0}$-semigroup $(S_{t})_{t\geq0}$ in $E$, explicitly given by
$$(S_{t} u_{0})(x)={\rm e}^{(\lambda_{1}-W(x))t}u_{0}(x).$$
For any $u_{0} \in E$ and $t > 0$ we have $(1 + W)S_{t}u_{0} \in E$, that is $S_{t}u_{0} \in D(\A_{0})$. In particular $(S_{t})_{t \geq 0}$ is an analytic semigroup on $E$ and $\|S_{t}\|_{L^2(\R^N) \to L^2(\R^N)} \leq {\rm e}^{\lambda_{1} t}$.
\end{lemma}

On the other hand, since $\BB$ is a bounded operator, we readily deduce the following expression of the semigroup $T_{t}$ in terms of the semigroup $S_{t}$ (see for instance K.J.~Engel \& R.~Nagel~\cite[Chapter III.1]{EN}).
Indeed, noting that \eqref{eq:evolution-A} can be written as
$$\frac{dv}{dt} = \A_{0}v(t) + Bv(t), \qquad v(0) = u_{0},$$
the solution $v(t)$ is given by the Duhamel formula
$$T_{t}u_{0} = v(t) = S_{t}u_{0} + \int_{0}^t S_{t - \tau}Bv(\tau)d\tau = S_{t}u_{0} + \int_{0}^t S_{t - \tau}BT_{\tau}u_{0}d\tau .$$
Analogously, the solution of
$$\frac{dz}{dt} = \A_{0}z(t) = Az(t) - Bz(t), \qquad z(0) = u_{0},$$
is given by
$$S_{t}u_{0} = z(t) = T_{t}u_{0} - \int_{0}^t T_{t - \tau}Bz(\tau) d\tau = 
T_{t}u_{0} - \int_{0}^t T_{t - \tau}B S_{\tau}u_{0} d\tau,
$$
so that we have also
$$T_{t}u_{0} = 
S_{t}u_{0} + \int_{0}^t T_{t - \tau}B S_{\tau}u_{0} d\tau. $$
We can thus state the following.

\begin{proposition}\label{prop:wellposedness}
The unbounded operator $\big(\A,D(\A)\big)$ generates a positive $C_{0}$-semigroup $(T_{t})_{t\geq0}$ in $E$, which yields the solutions of equation~\eqref{eq:evolution-A}.
For any $u_{0}\in D(\A)$ the mapping $t\mapsto T_{t}u_{0} =: v(t)$ is the unique classical solution of~\eqref{eq:evolution-A} and for all $u_0\in E$ it is the unique mild solution. Moreover the Duhamel formulas
\begin{equation}\label{eq:Duhamel}
v(t) = T_{t}u_{0} = S_{t}u_{0} + \int_{0}^t S_{t-\tau}\big(K*T_{\tau}u_{0}\big)\,d\tau,
\end{equation}
and
\begin{equation}\label{eq:Duhamel-bis}
v(t) = T_{t}u_{0} =  S_{t}u_{0} + \int_{0}^t T_{t-\tau}\big(K*S_{\tau}u_{0}\big)\,d\tau,
\end{equation}
hold for every $t \geq 0$ and $u_{0} \in E$.
\end{proposition}

It is noteworthy to observe that Proposition~\ref{prop:wellposedness} is still valid by replacing the choice of one of the above defined Banach spaces $E$ by the intersection $E_1\cap E_2$ of two such Banach spaces, endowed with the norm $\|\cdot\|_{E_1}+\|\cdot\|_{E_2}$.
The uniqueness property in this intersection then guarantees that, if $u_0\in E_1\cap E_2$, then the solutions in $E_1$ and $E_2$ coincide for all time.

\medskip

Now we return to the study of the spectral gap for the operator $\A$ in $L^2(\R^N)$, and we use the Duhamel formula to prove that the semigroup is quasi-compact.

\begin{lemma}\label{rouge}
The semigroup $(T_{t})_{t\geq0}$ is quasi-compact in $L^p(\R ^N)$ for $1 \leq p < \infty$,
and more precisely $\omega_\ess(A)\leq\lambda_1<0$.
\end{lemma}

\begin{proof}
For a given $u_{0} \in L^p(\R^N)$, by the Duhamel formula \eqref{eq:Duhamel-bis} we have
$$T_{t}u_{0} = S_{t} u_{0} + \int_{0}^t T_{t - \tau}\big(K*S_{\tau}u_{0}\big)\,d\tau .$$
We have $\|S_{t}\|_{L^p(\R^N)\to L^p(\R^N)}\leq {\rm e}^{\lambda_{1} t}<1$ for any $t>0$, hence setting
$$R_{t}u_{0} := \int_{0}^t T_{t - \tau}\big(K*S_{\tau}u_{0}\big)\,d\tau $$
it suffices to prove that the operator $R_{t}$ is compact for all $t$ large enough. As a matter of fact, it turns out that $R_{t}$ is compact for any $t > 0$.
To see this, we are going to use the Riesz-Fréchet-Kolmogorov theorem characterizing compact subsets of $L^p(\R^N)$ (see for instance K.~Yosida~\cite[Chapter X, section 1]{YosidaK}). 

First we check that, for any $\tau > 0$, the operator
$$u_{0} \mapsto K*S_{\tau}u_{0}$$ 
is compact on $L^p(\R^N)$.
Let $u_{0}\in L^p(\R^N)$ with $\|u_{0}\|_{L^p}\leq 1$. Observe first that 
$$\|K*S_{\tau}u_{0}\|_{L^p}\leq\|K\|_{L^1}\|S_{\tau}u_{0}\|_{L^p}\leq\sigma^2,$$
and thus the image of the unit ball of $L^p(\R^N)$ is bounded. Now, for $h \in \R^N$ define the translation operator $\tau_{h}$ by setting $\tau_{h}f = f(\cdot+h)$ for $f \in L^p(\R^N)$. We have 
$$\|\tau_{h}(K*S_{\tau}u_{0}) - K*S_{\tau}u_{0}\|_{L^p} = \|(\tau_{h} K - K)*S_{\tau}u_0\|_{L^p}\leq\|\tau_{h}K - K\|_{L^1}\xrightarrow[|h|\to0]{}0,$$
uniformly in $u_{0}$ in the unit ball of $L^p(\R^N)$. Next, by H\"older's inequality we have ($\frac 1p+\frac 1{p'}=1$)
\begin{align*}
|(K*S_{\tau}u_{0})(x)| & = 
\int_{\R^N} K(x-y)^{1/p'} K(x-y)^{1/p}{\rm e}^{(\lambda_{1} - W(y))\tau}u_{0}(y)\,dy\\
&\leq \left(\int_{\R^N} K(x-y) dy\right) ^{1/p'}\left(\int_{\R^N} K(x-y) {\rm e}^{p(\lambda_{1} - W(y))\tau}|u_{0}(y)|^p\,dy\right)^{1/p}\\
& \leq \sigma^{2/p'} \left(\int_{\R^N} K(x-y) \,{\rm e}^{p(\lambda_{1} - W(y))\tau}|u_{0}(y)|^p\,dy\right)^{1/p}.
\end{align*}
Hence by the Fubini-Tonelli theorem we may write (noting that $|x - y| \geq R/2$ when $|x| \geq R$ and $|y|\leq R/2$)
\begin{align*}
\int_{|x|\geq R}|(K*S_{\tau}u_{0})(x)|^pdx & \leq 
\sigma^{2p/p'} \int_{|x|\geq R}\int_{\R^N} K(x-y)\, {\rm e}^{p(\lambda_{1} - W(y))\tau}\vert u_{0}(y)\vert ^p\,dydx \\
&\leq \sigma^{2p/p'}  \int_{|x|\geq R} \int_{|y| < R/2} K(x-y)\, \vert u_{0}(y)\vert ^p \,dydx\\
& \hskip8mm + \sigma^{2p/p'}  \int_{|x|\geq R} \int_{|y|\geq R/2}K(x-y)\, {\rm e}^{-p\tau W(y)} \vert u_{0}(y)\vert ^p \,dydx\\
&\leq \sigma^{2p/p'} \int_{|z|\geq R/2} K(z)\,dz + \sigma^{2p}\sup_{|y|\geq R/2}{\rm e}^{-p\tau W(y)}\xrightarrow[R\to+\infty]{}0,
\end{align*}
uniformly in $u_{0}$ in the unit ball of $L^p(\R^N)$. Using the Riesz-Fréchet-Kolmogorov theorem we conclude that the mapping  $u_0\mapsto K*S_{\tau}u_{0}$ is compact.

Finally, since $T_{t-\tau}$ is a bounded operator, we infer that for any $0 < \epsilon \leq \tau \leq t$, the operators
$$u_{0}\mapsto T_{t-\tau}(K*S_{\tau}u_{0}) \quad\text{and}\quad
u_{0} \mapsto \int_{\epsilon}^t T_{t -\tau}(K*S_\tau u_{0})d\tau,$$
are compact operators on $L^p(\R^N)$. Since, as $\epsilon \to 0$ we have
$$\int_{\epsilon}^t T_{t -\tau}(K*S_\tau u_{0})d\tau \to \int_{0}^t T_{t -\tau}(K*S_\tau u_{0})d\tau = R_{t}u_{0},$$
uniformly on the unit ball of $L^p(\R^N)$, we conclude that $R_{t}$ is compact.
\end{proof}

Now we can state our convergence result.

\begin{corollary} \label{etleclochard}
There exist $C,a>0$ such that, for all $u_{0} \in L^2(\R^N)$ and all $t \geq 0$, we have
\begin{equation}\label{eq:conv-L2}
\left\|T_{t}u_{0} - \langle u_{0},\phi\rangle \phi \right\|_{L^2}\leq C\, {\rm e}^{-at}\, \left\|u_{0}-\langle u_{0},\phi\rangle\phi\right\|_{L^2}.
\end{equation}
If additionally $b_\epsilon>\bar\Phi:=\sup_{[0,\infty)}\Phi$, where $b_\epsilon$ and $\Phi$ are defined in~\eqref{eq:b-eps} and~\eqref{def:Phi} respectively,
then one can choose any $a<a_*=:b_\epsilon-\bar\Phi$.
\end{corollary}

\begin{proof}
We proved that $\omega_\ess(A)\leq\lambda_1<0$, and we have the identity $\omega_\ess(A_{|(\R\varphi)^\perp})=\omega_\ess(A)$ (use for instance~\cite[Proposition IV.2.12]{EN}).
We deduce that for any $w\in(\lambda_1,0)$ the set $\sigma(A_{|(\R\varphi)^\perp})\cap[w,0]$ is finite and made only of eigenvalues.
Since the kernel of $A$ is the space generated by $\varphi$, zero is not an eigenvalue of $A_{|(\R\varphi)^\perp}$ and consequently there exists $a>0$ such that $\sigma(A_{|(\R\varphi)^\perp})\subset(-\infty,-a)$.
This implies that
\[\omega_0(A_{|(\R\varphi)^\perp})=\s(A_{|(\R\varphi)^\perp})<-a\]
and accordingly the existence of $C>0$ such that for all $u_0\in(\R\varphi)^\perp$
\[\left\|T_{t}u_{0}\right\|_{L^2}\leq C\, {\rm e}^{-at}\, \left\|u_{0}\right\|_{L^2}.\]
For $u_0\not\in(\R\varphi)^\perp$, applying this stability result to $u_{0}-\langle u_{0},\phi\rangle\phi\in(\R\varphi)^\perp$ gives~\eqref{eq:conv-L2} since $T_t\varphi=\varphi$ for all $t\geq0$.

For proving the second part of Corollary~\ref{etleclochard}, it suffices to check that, if $b_\epsilon>\bar\Phi$, then there is no eigenvalue of $A$ in the interval $(\bar\Phi-b_\epsilon,0)$.
Corollary~\ref{cor:spectral-gap} ensures that there is no non-zero eigenvalue above $\bar\Phi+\lambda_1$, and from Theorem~\ref{thm:eigen} we know that $\lambda_1\leq -b_\epsilon$.
So the result is proved.
\end{proof}

\subsection{Analysis in $C_{0}(\R^N)$ and $\MM(\R^N)$}

We start by checking that $(T_{t})_{t\geq0}$ is quasi-compact in $C_{0}(\R^N).$

\begin{lemma}
The semigroup $(T_{t})_{t\geq0}$ is quasi-compact in $C_{0}(\R^N).$
\end{lemma}

\begin{proof}
Using the Duhamel formula~\eqref{eq:Duhamel-bis} in a similar way as we did in Lemma~\ref{rouge}, we only have to prove that  $u_{0} \mapsto K*S_{\tau}u_{0}$ is compact for any $\tau > 0$. 
This property is a consequence of Ascoli's theorem: indeed we have, for any $u_{0} \in C_{0}(\R^N)$ with $\|u_0\|_{L^\infty} \leq 1$,
$$\|\tau_{h}(K*S_{\tau}u_0)-K*S_{\tau}u_0\|_{L^\infty} \leq \|\tau_{h} K - K\|_{L^1}\xrightarrow[|h|\to0]{}0,$$
and 
$$|K*S_{\tau}u_0|\leq K*{\rm e}^{-\tau W}\in C_{0}(\R^N),$$
uniformly in $u_{0}$ with $\|u_{0}\|_{L^\infty} \leq 1$. 
The proof of the lemma is complete.
\end{proof}

Contrary to the $L^2$ case, we cannot argue through the orthogonal space of $\R\varphi$.
Yet, the positivity of $\varphi$ combined to the quasi-compactness of $(T_t)_{t\geq0}$ is enough to prove the following results.

\begin{corollary}\label{villoise}
The eigenfunction $\phi$ belongs to $L^1(\R^N),$ and there exist $C,a>0$ such that, for all $u_{0}\in C_{0}(\R^N)$ and all $t \geq 0$,
$$\left\|T_{t}u_{0}-\langle u_{0},\phi\rangle\phi\right\|_{L^\infty}\leq C\, {\rm e}^{-at}\left\|u_{0}-\langle u_{0},\phi\rangle\phi\right\|_{L^\infty}.$$
\end{corollary}

\begin{proof}
We have proved that $(T_{t})_{t\geq0}$ is quasi-compact and we know that $0$ is an eigenvalue of $\A$ associated to a strictly positive eigenfunction $\phi$.
We deduce from Corollary B-IV-2.11 in~\cite{Nagel86} that there exists a positive projection $\PP$ of finite rank and suitable constants $C,a>0$ such that, for all $t \geq 0$,
$$\|T_{t} - \PP\|_{C_{0}(\R^N) \to C_{0}(\R^N)}\leq C \, {\rm e}^{-at}.$$
Let us now identify this projection.
From Corollary~\ref{etleclochard} we deduce that for all $u_0\in C_c(\R^N)$, $\PP u_0=\langle u_0,\varphi\rangle\varphi$.
Since $\PP$ is a projection, this implies that for all $u_0\in C_c(\R^N)$, $|\langle u_0,\varphi\rangle|\leq{\|u_0\|}_{L^\infty}/{\|\varphi\|}_{L^\infty}$.
Consequently $\varphi$ belongs to $L^1(\R^N)$, and the bounded operator $u_0\mapsto \langle u_0,\varphi\rangle\varphi$ on $C_0(\R^N)$ coincides with $\PP$ on the dense subset $C_c(\R^N)$.
So they are necessarily equal on $C_0(\R^N)$ and the proof is complete.
\end{proof}

\begin{corollary}\label{epoque}
There exist $C,a>0$ such that, for all $\mu\in \MM(\R^N)$ and all $t\geq0$,
$$\left\|T_{t}\mu - \langle\mu,\phi\rangle\phi\right\|_\TV\leq C{\rm e}^{-at}\left\|\mu-\langle \mu,\phi\rangle\phi\right\|_\TV.$$
\end{corollary}

Notice that this implies also the exponential convergence in $L^1(\R^N)$, since for $u_{0} \in L^1(\R^N)$ and $\mu := u_{0}(x)dx$ we have ${\|\mu\|}_\TV = {\|u_0\|}_{L^1}$.

\begin{proof}
Due to the duality $\MM(\R^N) = (C_{0}(\R^N))'$ and the definition of the total variation norm as a duality norm
$$\left\|\mu\right\|_\TV=\sup_{f\in C_{0},\,{\|f\|}_{L^\infty}\leq1}\langle \mu,f\rangle,$$
the result is a  consequence of Corollary~\ref{villoise} applied to the dual semigroup ${(T_{t}^*)}_{t\geq0} = {(T_{t})}_{t \geq 0}$.
\end{proof}

\subsection{Study in $L^p(\R^N)$ with $1\leq p<\infty$}

We have proved that $\phi\in L^1(\R^N)\cap C_{0}(\R^N),$ so that $\phi\in L^p(\R^N)$ for all $p\in[1,\infty].$
Also recall that in Lemma~\ref{rouge} we have shown that $(T_{t})_{t\geq0}$ is quasi-compact in $L^p(\R^N)$ for any $p\in[1,\infty)$. 

\begin{corollary}
Let $p\in[1,\infty)$. There exist $C,a>0$ such that, for all $u_{0}\in L^p(\R^N)$ and all $t\geq0$,
$$\left\|T_{t}u_{0} - \langle u_{0},\phi\rangle\phi\right\|_{L^p}\leq C{\rm e}^{-at}\left\|u_{0}-\langle u_{0},\phi\rangle\phi\right\|_{L^p}.$$
\end{corollary}

\begin{proof} For $u_{0} \in L^p(\R^N)$ denote by $\PP u_{0} := \langle u_{0},\phi \rangle\phi$. We know that there exist two constants $C > 0$ and $a > 0$ such that
$$\|T_{t} - \PP\|_{L^1(\R^N) \to L^1(\R^N)} \leq C {\rm e}^{-at}
\quad\text{and} \quad
\|T_{t} - \PP\|_{L^2(\R^N) \to L^2(\R^N)} \leq C {\rm e}^{-at}.
$$
Therefore, for $1 < p < 2$, by interpolation (see for instance L.~Tartar~\cite[Chapter 21, Theorem 21.2]{TartarL-Sobolev-Interpolation}) we have
$$\|T_{t} - \PP\|_{L^p(\R^N) \to L^p(\R^N)} \leq \|T_{t} - \PP\|_{L^1(\R^N) \to L^1(\R^N)}^{1 - \theta} \|T_{t} - \PP\|_{L^2(\R^N) \to L^2(\R^N)}^\theta,$$
where $\theta \in (0,1)$ is defined by $1/p = (1 - \theta) + (\theta/2) = 1 - (\theta/2)$. Thus we have
$$\|T_{t} - \PP\|_{L^p(\R^N) \to L^p(\R^N)} \leq C\, {\rm e}^{-at}.$$
When $2 < p < \infty$, then $p' := p/(p - 1) \in (1,2)$, and since $T_{t}^* = T_{t}$ and $\PP^* = \PP$ on the subspace $L^1(\R^N) \cap C_{0}(\R^N)$, which is dense both in $L^p(\R^N)$ and $L^{p'}(\R^N)$, the above inequality applied to $p'$ shows that
$$\|T_{t} - \PP\|_{L^p(\R^N) \to L^p(\R^N)} = \|T_{t}^* - \PP^*\|_{L^{p'}(\R^N) \to L^{p'}(\R^N)}\leq C\, {\rm e}^{-at},$$
which concludes the proof.
\end{proof}


\section{Long time asymptotics of the replicator-mutator model}\label{sec:rep-mut}

We begin by giving the definition of what we shall call classical and mild solutions for the nonlinear replicator-mutator equation \eqref{eq:rep-mut}. Let us denote by $E_{+}$ the positive cone of the Banach lattice $E$ and define
$$E(W):=\{u\in E: Wu\in E\}$$
endowed with the norm
$$\left\|u\right\|_{E(W)}:={\|u\|}_E+{\|Wu\|}_E.$$
Notice that in the framework of Section~\ref{sec:asympto}, this space is nothing but the domain of the operator $\A_{0}$ endowed with the graph norm.

\begin{definition}
A function $u:[0,+\infty)\dans E_{+}$ is called a \emph{classical solution} of Equation~\eqref{eq:rep-mut} if $u\in C^1([0,+\infty),E)\cap C([0,+\infty),E(W))$ and~\eqref{eq:rep-mut} holds.

It is called a \emph{mild solution} of Equation~\eqref{eq:rep-mut} if $u\in C([0,\tau],E)\cap L^1([0,\tau],E(W))$ for all $\tau>0$, and
$$u(t)={\rm e}^{-\lambda_{*} t}T_{t}u_{0} + \int_{0}^t\langle u(s),W\rangle\, {\rm e}^{-\lambda_{*} (t-s)}T_{t-s}u(s)\,ds,$$
for all $t\geq0$ (Recall that the eigenvalue $\lambda_{*}$ is defined in \eqref{eq:Def-lambda-*}).
\end{definition}

\begin{proposition}[The solution of \eqref{eq:rep-mut} in terms of that of \eqref{eq:parabolique}]\label{prop:wellposed_rep-mut}
Suppose that Assumptions~\ref{ass:kernel} and~\ref{ass:potential} are verified, and let $u_{0}\in E_{+}$ with $\langle u_{0},\1\rangle=1.$
There exists a unique mild solution to Equation~\eqref{eq:rep-mut} starting from $u_{0}$, and it is given by
$$u(t)=\frac{T_{t}u_{0}}{\langle T_{t}u_{0},\1\rangle}.$$
If additionally $u_{0}\in E(W)$ then it is a classical solution.
\end{proposition}

Before giving the proof of this result, note that the condition $\langle u_{0},\1\rangle=1$ implies that when $E=L^p(\R^N)$, $u_0$ also belongs to $L^1(\R^N)$.
By virtue of the comment after Proposition~\ref{prop:wellposedness}, this guarantees that $T_tu_0$ is also in $L^1(\R^N)$ for all $t\geq0$, and consequently $\langle T_{t}u_{0},\1\rangle$ is finite.

\begin{proof}
We start with the case $u_{0}\in E(W)$ and show that the function
$$u(t)=\frac{T_{t}u_{0}}{\langle T_{t}u_{0},\1\rangle}$$
is a classical solution to Equation~\eqref{eq:rep-mut}.
 By Proposition~\ref{prop:wellposedness}, since $u_{0}\in E(W)=D(\A)$, the function $t\mapsto T_{t}u_{0}$ is continuously differentiable and we have
$$\frac{d}{dt} u(t)=\frac{1}{\langle T_{t}u_{0},\1\rangle}\bigg[\A T_{t}u_{0}-\frac{\langle\A T_{t}u_{0},\1\rangle}{\langle T_{t}u_{0},\1\rangle}T_{t}u_{0}\bigg].$$
Using the fact that $\A$ is self-adjoint and that
$$\langle\A T_{t}u_{0},\1\rangle=\langle T_{t}u_{0},\A\1\rangle = \lambda_{*} \langle T_{t}u_{0},\1\rangle-\langle T_{t}u_{0},W\rangle$$
we get
$$\frac{d}{dt} u(t) = (\A - \lambda_{*} I) u(t)+\langle u(t),W\rangle u(t).$$
This proves the existence part.

For the uniqueness of the solution, we use the uniqueness result for the linear equation.
Let $u$ be a classical solution to Equation~\eqref{eq:rep-mut}.
It is clear that
$$v(t)=u(t)\, \exp\left(\lambda_{*} t - \int_{0}^t\langle u(s),W\rangle\,ds \right)$$
is continuously differentiable, and by differentiation we readily get that $v$ is the unique classical solution to Equation~\eqref{eq:evolution-A}.

\medskip

Let us now turn to the case where $u_{0}$ does not necessarily belong to $E(W)$.
The approach is the same as before, but one has to deal with mild solutions.
Defining ${\widetilde T}_{t}={\rm e}^{-\lambda_{*} t}T_{t}$ and setting again 
$$u(t)=\frac{T_{t}u_{0}}{\langle T_{t}u_{0},\1\rangle}=\frac{{\widetilde T}_{t}u_{0}}{\langle {\widetilde T}_{t}u_{0},\1\rangle},$$
we have for any $t \geq 0$
\begin{equation}\label{machin-chose}
\int_{0}^t\langle u(s),W\rangle {\widetilde T}_{t-s}u(s)\,ds=\int_{0}^t\frac{\langle {\widetilde T}_{s}u_{0},W\rangle}{\langle {\widetilde T}_{s}u_{0},\1\rangle}\frac{{\widetilde T}_{t}u_{0}}{\langle{\widetilde T}_{s}u_{0},\1\rangle}ds=\bigg(\int_{0}^t\frac{\langle {\widetilde T}_{s}u_{0},W\rangle}{\langle {\widetilde T}_{s}u_{0},\1\rangle^2}ds\bigg){\widetilde T}_{t}u_{0}.
\end{equation}
Since $t\mapsto T_{t}u_{0}$ is a mild solution of Equation~\eqref{eq:evolution-A} we get by integration, using Fubini-Tonelli's theorem (note that $T_{t}u_{0}$ and $W$ are nonnegative),
$$\langle T_{t}u_{0},\1\rangle=\langle u_{0},\1\rangle+\left\langle\int_{0}^t T_{s}u_{0}\,ds,\A\1\right\rangle=1+\int_{0}^t\langle T_{s}u_{0},\lambda_{*} -W\rangle\,ds.$$
Since $t\mapsto\langle T_{t}u_{0},\lambda_{*} - W\rangle$ is locally integrable, it ensures that $t\mapsto\langle T_{t}u_{0},\1\rangle$ belongs to $W^{1,1}_{\rm loc}(0,+\infty)$ with,  in the weak sense,
$$\frac{d}{dt}\langle T_{t}u_{0},\1\rangle=\langle T_{t}u_{0},\lambda_{*} -W\rangle,\quad\text{or equivalently}\quad
\frac{d}{dt}\langle{\widetilde T}_{t}u_{0},\1\rangle=-\langle{\widetilde T}_{t}u_{0},W\rangle.$$
As a result, since $\langle u_{0},\1\rangle=1$,
$$\int_{0}^t\frac{\langle{\widetilde T}_{s}u_{0},W\rangle}{\langle{\widetilde T}_{s}u_{0},\1\rangle^2}ds=\frac{1}{\langle{\widetilde T}_{t}u_{0},\1\rangle}-1$$
which, combined with \eqref{machin-chose}, yields
$$\int_{0}^t\langle u(s),W\rangle{\widetilde T}_{t-s}u(s)\,ds=\frac{{\widetilde T}_{t}u_{0}}{\langle{\widetilde T}_{t}u_{0},\1\rangle}-{\widetilde T}_{t}u_{0} = u(t)-{\widetilde T}_{t}u_{0}$$
and this exactly means that $u$ is a mild solution to Equation~\eqref{eq:rep-mut}.

To prove the uniqueness of the solution, we consider a mild solution $u$ to Equation~\eqref{eq:rep-mut} and we define 
$$v(t)=u(t)\, \exp\left(\lambda_{*} t-\int_{0}^t\langle u(s),W\rangle\,ds\right).$$
Let us also take a function $f\in C_{c}(\R^N)\subset D(\A)$.
Since $T_{t}$ is self-adjoint, by differentiation of the equality
\begin{align*}
\langle u(t),f\rangle&=\langle{\widetilde T}_{t}u_{0},f\rangle+\int_{0}^t\langle u(s),W\rangle\langle{\widetilde T}_{t-s} u(s),f\rangle\,ds\\
&=\langle u_{0},{\widetilde T}_{t} f\rangle+\int_{0}^t\langle u(s),W\rangle\langle u(s),{\widetilde T}_{t-s}f\rangle\,ds
\end{align*}
we get
\begin{align*}
\frac{d}{dt}\langle u(t),f\rangle & = \langle u_{0},{\widetilde T}_{t}(\A - \lambda_{*} ) f\rangle+\langle u(t),W\rangle\langle u(t),f\rangle+\int_{0}^t\!\langle u(s),W\rangle\langle u(s),{\widetilde T}_{t-s}^*(\A -\lambda_{*} ) f\rangle\,ds\\
&=\langle u(t),(\A-\lambda_{*} ) f\rangle+\langle u(t),W\rangle\langle u(t),f\rangle.
\end{align*}
Since $t\mapsto\int_{0}^t\langle u(s),W\rangle\,ds$ belongs to $W^{1,1}_{\rm loc}(0,+\infty)$, we obtain
\begin{align*}
\frac{d}{dt}\langle v(t),f\rangle&=\frac{d}{dt}\langle u(t),f\rangle {\rm e}^{\lambda_{*} t-\int_{0}^t\langle u(s),W\rangle\,ds}
+\big(\lambda_{*} -\langle u(t),W\rangle \big)\langle u(t),f\rangle {\rm e}^{\lambda_{*} t-\int_{0}^t\langle u(s),W\rangle\,ds}\\
&=\langle u(t),\A f\rangle {\rm e}^{\lambda_{*} t-\int_{0}^t\langle u(s),W\rangle\,ds}=\langle v(t),\A f\rangle.
\end{align*}
Integrating between $0$ and $t$ we obtain
$$\langle v(t),f\rangle-\langle u_{0},f\rangle=\int_{0}^t\langle v(s),\A f\rangle\,ds=\left\langle\int_{0}^t v(s)\,ds,\A f\right\rangle=\left\langle\A\int_{0}^t v(s)\,ds, f\right\rangle$$
for all $f\in C_{c}(\R^N)$.
By density of $C_{c}(\R^N)$ in $E$ we get that $v$ is a mild solution to the linear equation, and so $v(t)=T_{t}u_{0}$. 
\end{proof}

Due to the explicit expression of the solution to the replicator-mutator model in terms of the semigroup $(T_{t})_{t\geq0}$ obtained in Proposition~\ref{prop:wellposed_rep-mut}, the conclusion of Theorem~\ref{thm:rep-mut} follows from Theorem~\ref{thm:asympto}.

\begin{proof}[Proof of Theorem~\ref{thm:rep-mut}]
It suffices to write
\begin{align*}
\left\|\frac{T_{t}u_{0}}{\langle T_{t}u_{0},\1\rangle}-\frac{\phi}{\langle\phi,\1\rangle}\right\|_E&=\left\|\frac{\langle\langle u_{0},\phi\rangle\phi - T_{t}u_{0},\1\rangle T_{t}u_{0} + \langle T_{t}u_{0},\1\rangle(T_{t}u_{0} - \langle u_{0},\phi\rangle\phi)}{\langle T_{t}u_{0},\1\rangle\langle\phi,\1\rangle\langle u_{0},\phi\rangle}\right\|_E\\
&\leq\frac{\|T_{t}u_{0}\|_E}{\langle T_{t}u_{0},\1\rangle}\frac{\|T_{t}u_{0} - \langle u_{0},\phi \rangle\phi\|_\TV}{\langle\phi,\1\rangle\langle u_{0},\phi \rangle}+\frac{\|T_{t}u_{0} - \langle u_{0},\phi \rangle\phi\|_E}{\langle\phi,\1\rangle\langle u_{0},\phi \rangle}
\end{align*}
and use the exponential convergence in Theorem~\ref{thm:asympto} to get the result which is valid for any~$E$. For $E=\MM(\R^N)$ or $E=L^1(\R^N)$ the stronger conclusion follows by noticing that in these two cases $\langle T_{t}u_{0},\1\rangle=\|T_{t}u_{0}\|_E$.
\end{proof}

\section{About the non-symmetric case}\label{sec:non-sym}

In this last section, we do not assume that $J$ is even, that is Assumption~\ref{ass:J-symmetric} is not supposed any more.
In counterpart, we strengthen Assumption~\ref{ass:linkJW} by requiring that Assumption~\ref{ass:Coville} is verified.

As in the symmetric case, we first consider the Banach space $E=L^2(\R^N)$.
For $J$ non-even, the operator $-L$ is not self-adjoint, and we cannot use the variational approach to prove the existence of a first eigenfunction.
We bypass this issue by taking advantage of the fact that the semigroup generated by $-L$ is irreducible.

Proposition~\ref{prop:wellposedness} does not require $J$ to be even, and it ensures that $-L$ generates a positive semigroup $(U_t)_{t\geq0}$ in any of the considered Banach spaces $E$. 
The strong maximum principle in Lemma~\ref{lem:Max-principle} is also valid without evenness assumption on~$J$,
and it precisely means that $(U_t)_{t\geq0}$ is irreducible in $L^p(\R^N)$ for $1\leq p<\infty$ (see~\cite[Definition C-III-3.1.(v)]{Nagel86}, recalling that the quasi-interior points in $L^p(\R^N)$ with $1\leq p<\infty$ are the functions strictly positive a.e.).

We thus have the following result.

\begin{theorem}[The non-symmetric case]\label{thm:L2-non-sym}
Let Assumptions~\ref{ass:kernel}, \ref{ass:potential} and~\ref{ass:Coville} hold.
Then there exist $\lambda_1<0$ and two positive functions $\varphi$ and $\varphi^*$ in $L^2(\R^N)$ with $\|\varphi\|_{L^2(\R^N)}=\langle\varphi^*,\varphi\rangle=1$ such that
\[L\varphi=\lambda_1\varphi\qquad\text{and}\qquad L^*\varphi^*=\lambda_1\varphi^*.\]
Moreover, there exist two constants $C,a>0$ such that, for any $u_0\in L^2(\R^N)$ and all $t\geq0$
\begin{equation}\label{eq:conv-L2-non-sym}
\big\|{\rm e}^{\lambda_1t}\,U_tu_0-\langle\varphi^*,u_0\rangle\varphi\big\|_{L^2}\leq C{\rm e}^{-at}\left\|u_0-\langle\varphi^*,u_0\rangle\varphi\right\|_{L^2}.
\end{equation}
\end{theorem}

Let us mention that the convergence~\eqref{eq:conv-L2-non-sym} ensures the uniqueness of the triplet $(\lambda_1,\varphi,\varphi^*)$.

\begin{proof}
First we prove that Assumption~\ref{ass:Coville} ensures that $\omega_0(-L)>0$.
Take $\epsilon,\eta>0$ and $B\subset\R^N$ such that
\[\sigma^2\essinf_{B_\epsilon}\int_{B_\epsilon}\frac{J(x-y)}{W(y)}dy\geq1+\eta\]
where we recall the notation $B_\epsilon=B\cap[W\geq\epsilon]$,  and define the function
\[\psi(x):=\frac{1}{W(x)}\1_{B_\epsilon}(x).\]
Clearly $B_\epsilon$ is necessarily essentially bounded, so that $\psi\in D(-L)=L^2(1+W)$, and for almost all $x\in B_\epsilon$ we have
\[-L\psi(x)=\sigma^2\int_{B_\epsilon}\frac{J(x-y)}{W(y)}dy-1\geq\eta.\]
Since $\psi\leq1/\epsilon$, we deduce that
\[-L\psi\geq\epsilon\eta\,\psi,\quad\text{and consequently}\quad  U_t\psi\geq{\rm e}^{\epsilon\eta t}\psi.\]
This ensures that $\omega_0(-L)\geq\epsilon\eta>0$, and we set $\lambda_1:=-\omega_0(-L)$.

Besides, the proof of Lemma~\ref{rouge} does not require the evenness of $J$, and it guarantees that $\omega_\ess(\lambda_1-L)\leq\lambda_1$.
We thus have $\omega_\ess(\lambda_1-L)<\omega_0(\lambda_1-L)=0$ and this has two implications.
First, the semigroup $({\rm e}^{\lambda_1 t}\,U_t)_{t\geq0}$ generated by $\lambda_1-L$ is quasi-compact and second,
since $\omega_0(\lambda_1-L)=\max(\omega_\ess(\lambda_1-L),\s(\lambda_1-L))$ (see for instance K.J. Engel \& R. Nagel~\cite[Corollary IV.2.11]{EN}), the spectral bound $\s(\lambda_1-L)$ is zero.

We are in position to apply the result in W. Arendt \& al.~\cite[Section~C-IV, Remark~2.2.(d)]{Nagel86} to the semigroup $(T_t)_{t\geq0}:=({\rm e}^{\lambda_1 t}\,U_t)_{t\geq0}$, and it guarantees the existence of $\varphi$ and $\varphi^*$ positive such that~\eqref{eq:conv-L2-non-sym} holds.
Using this convergence to pass to the limit $s\to+\infty$ in $T_tT_su_0=T_{t+s}u_0$ we get that $T_t\varphi=\varphi$, which yields $L\varphi=\lambda_1\varphi$ and, choosing $u_0=\varphi$ in~\eqref{eq:conv-L2-non-sym}, $\langle\varphi^*,\varphi\rangle=1$.
Finally, taking $Lu_0$ in place of $u_0$ in~\eqref{eq:conv-L2-non-sym} we get by passing to the limit $t\to+\infty$ that
\[\langle\varphi^*,u_0\rangle L\varphi=\langle\varphi^*,Lu_0\rangle\varphi\]
and so $\lambda_1\langle\varphi^*,u_0\rangle=\langle\varphi^*,Lu_0\rangle=\langle L^*\varphi^*,u_0\rangle$ for all $u_0\in D(L)=L^2(1+W)$.
This obviously implies that $L^*\varphi^*=\lambda_1\varphi^*$, and the proof is complete.
\end{proof}

Now that we have Theorem~\ref{thm:L2-non-sym} at hand, we can use the same strategy as in the symmetric case to deduce the following counterpart of Theorem~\ref{thm:asympto} in the non-symmetric case.

\begin{corollary}
The functions $\varphi$ and $\varphi^*$ belong to $L^1(\R^N)\cap C_0(\R^N)$ and there exist two constants $C > 0$ and $a > 0$ such that, for any $u_{0}\in E$, the solution $u=u(t,x)$ of \eqref{eq:parabolique} starting from $u_{0}=u_0(x)$ satisfies, for all $t > 0$,
$$
\big\| {\rm e}^{\lambda_{*} t}u(t,\cdot)-\langle \phi^*,u_{0} \rangle \phi\big\|_{E} \leq C\, {\rm e}^{-at}\left\|u_{0} - \langle \varphi^*, u_{0} \rangle\phi\right\|_{E} ,
$$
where $\lambda_*=\lambda_1+\sigma^2$.
\end{corollary}

\medskip
\appendix
\section{Comparison between Assumptions~\ref{ass:linkJW} and~\ref{ass:Coville}}\label{app:ex1}

First we check that when $\meas[W=0]=0$, Assumption~\ref{ass:Coville} is equivalent to the existence of a Borel set $B\subset\R^N$ such that
\begin{equation}\label{eq:CovilleB}
\sigma^2\essinf_{x\in B}\int_{B}\frac{J(x-y)}{W(y)}dy>1,
\end{equation}
which is the condition appearing in~\cite[Theorem~2.1]{Li-Cov-Wan}.
Clearly Assumption~\ref{ass:Coville} implies the existence of such a $B$ (take $B=B_\epsilon$).
Suppose now that $\meas[W=0]=0$ and that there exists $B$ such that~\eqref{eq:CovilleB} is verified.
For all $\epsilon>0$ we have
\[\sigma^2\essinf_{x\in B_\epsilon}\int_{B_\epsilon}\frac{J(x-y)}{W(y)}dy\geq\sigma^2\essinf_{x\in B}\int_{B_\epsilon}\frac{J(x-y)}{W(y)}dy\]
since $B_\epsilon=B\cap[W\geq\epsilon]\subset B$, and the monotone convergence theorem ensures that
\[\int_{B_\epsilon}\frac{J(x-y)}{W(y)}dy\xrightarrow[\epsilon\to0]{}\int_{B}\frac{J(x-y)}{W(y)}dy\]
because $\meas(B\cap[W=0])=0$.
So~\eqref{eq:CovilleB} guarantees the existence of $\epsilon>0$ such that
\[\sigma^2\essinf_{x\in B_\epsilon}\int_{B_\epsilon}\frac{J(x-y)}{W(y)}dy>1.\]

\medskip

The second part of this appendix section is devoted to the proof of the claim in Example~\ref{ex:1}.
Consider the one dimensional space $\R^N=\R$ and the coefficients
\[J(z)=\frac12\1_{[-1,1]}(z)\qquad\text{and}\qquad W(x)=\sqrt{|x|}.\]

First we check that Assumption~\ref{ass:Coville} is verified if and only if $\sigma^2>\frac 1{\sqrt 2}$.
Since $\meas[W=0]=0$, Assumption~\ref{ass:Coville} is equivalent to the existence of a Borel set $B\subset\R$ such that~\eqref{eq:CovilleB} is verified.
Moreover, since $\int_{B}\frac{J(x-y)}{W(y)}dy\leq \frac 12\int_{x-1}^{x+1}\frac{1}{\sqrt {\vert y\vert}}dy\to 0$ as $\vert x\vert \to \infty$, such a $B$ has to be essentially bounded, and we set
\[R:=\inf\{C>0\,:\, |x|\leq C\ \text{for almost all}\ x\in B\}\in[0,+\infty).\]
Now let us define $f:\R\to\R$ by
\[f(x):=\int_{-R}^{R}\frac{J(x-y)}{W(y)}dy.\]
This function is positive, continuous, even, and non-increasing on $[0,+\infty)$.
The definition of~$R$ thus implies that $\essinf_{x\in B}f(x)=f(R)$, and from \eqref{eq:CovilleB} we deduce that
\begin{equation}\label{eq:f-sigma}
\sigma ^2 f(R)>1.
\end{equation}
But $f(R)$ can be computed explicitly as follows
\[f(R)=\left\{\begin{array}{lcl}
2\sqrt{R}&&\text{if}\ 0\leq R\leq \frac 12,\vspace{1mm}\\
\sqrt{R}+\sqrt{1-R}&&\text{if}\ \frac 12\leq R\leq1,\vspace{1mm}\\
\sqrt{R}-\sqrt{R-1}&&\text{if}\ R\geq1,
\end{array}\right.\]
and we deduce that $\max_{R\geq0}f(R)=f(1/2)=\sqrt{2}$.
As a result~\eqref{eq:f-sigma} enforces $\sigma^2>\frac 1{\sqrt{2}}$.
Reciprocally, if $\sigma^2>\frac 1{\sqrt{2}}$, then Assumption~\ref{ass:Coville} is verified by choosing $B=[-\frac12,\frac12]$.

Now we check that Assumption~\ref{ass:linkJW} is verified as soon as  $\sigma^2>\frac{4}{4+\pi}$.
Since $1/W$ is locally integrable and $J$ is locally supported, it suffices to verify that
\[\sigma^2\iint_{B\times B}\frac{J(x-y)}{W(x)W(y)}\,dxdy>\int_{B}\frac{1}{W(x)}\,dx\]
for some Borel set $B\subset\R$, {\it i.e.} we can take $\epsilon=0$ in Assumption~\ref{ass:linkJW}.
We choose $B=[-1,1]$ and explicit computations yield
\begin{align*}
\iint_{B\times B}\frac{J(x-y)}{W(x)W(y)}\,dxdy&=\int_0^1\frac{1}{\sqrt{x}}\int_{x-1}^1\frac{1}{\sqrt{|y|}}\,dy\,dx\\
&=\int_0^1\frac{2}{\sqrt{x}}(1+\sqrt{1-x})\,dx=4+\pi
\end{align*}
and
\[\int_{B}\frac{1}{W(x)}\,dx=2\int_{0}^1\frac{1}{\sqrt{x}}\,dx=4.\]
As a result it suffices that $\sigma^2(4+\pi)>4$ for Assumption~\ref{ass:linkJW} to be verified.


\section{Proof of the estimate of Example~\ref{ex:2}}\label{app:ex2}

Consider the one dimensional space $\R^N=\R$ and 
\[J(z)=\frac14\1_{[-2,2]}(z)\qquad\text{and}\qquad W(x)=|x|^m\ \text{with}\ m>1.\]
We want to derive, for some well-chosen sets $B_\epsilon$ and $\Omega$, an estimate on the quantity
\[a_*=b_\epsilon-\bar\Phi\]
where
\[b_\epsilon=\bigg(\sigma^2\iint_{B_\epsilon\times B_\epsilon}\frac{J(x-y)}{W(x)W(y)}\,dxdy-\int_{B_\epsilon}\frac{dx}{W(x)}\bigg)\bigg(\int_{B_\epsilon}\frac{dx}{W(x)}\bigg)^{-2}\]
and $\bar\Phi$ is the maximal value of the function
\[\Phi(\xi)=\min\Big\{\sigma^2-\xi,\eta\meas(\Omega)+a_1\xi+a_2\sqrt{\xi}\Big\}\]
with
\[\eta=\esssup_{2\Omega}K-\essinf_{2\Omega}K,\]
\[a_1=\Big(\essinf_{2\Omega}K\Big)\bigg(\int_{\Omega^c}\frac1W\bigg)\qquad\text{and}\qquad a_2=2\sigma\sqrt{\sup_{x\in\R}\int_{\Omega^c}\frac{K(x-y)}{W(y)}dy}.\]

Let us start with $b_\epsilon$. Choosing $B_\epsilon=[-1,1]\cap[W\geq\epsilon]$ we have for all $\epsilon>0$
\[\iint_{B_\epsilon\times B_\epsilon}\frac{J(x-y)}{W(x)W(y)}\,dxdy=\frac14\bigg(\int_{B_\epsilon}\frac{dx}{W(x)}\bigg)^2,\]
so that
\begin{equation}\label{eq:b-eps-ex}
b_\epsilon=\frac{\sigma^2}{4}-\bigg(\int_{B_\epsilon}\frac{dx}{W(x)}\bigg)^{-1}\xrightarrow[\epsilon\to0]{}\frac{\sigma^2}{4}.
\end{equation}


Now let us estimate $\bar\Phi$ for the set $\Omega=(-1,1)$.
We have
\[\eta=0,\]
\[a_1=\frac{\sigma^2}{4}\int_{\Omega^c}|x|^{-m}dx=\frac{\sigma^2}{2}\int_1^\infty x^{-m}dx=\frac{\sigma^2}{(m-1)}\]
and
\[a_2=\sigma^2\sqrt{\sup_{x\in\R}\int_{\Omega^c}\1_{[-2,2]}(x-y)|y|^{-m}dy}=\sigma^2\sqrt{2\int_1^5 y^{-m}dy}=\sigma^2\sqrt2\sqrt{\frac{1-5^{1-m}}{m-1}}.\]
It is easy to see from the definition of $\Phi$ that
\[\bar\Phi\leq \eta\meas(\Omega)+a_1\sigma^2+a_2\sigma.\]
We thus get the estimate
\[\bar\Phi\leq\sigma^2\left(\frac{\sigma^2}{m-1}+\sqrt 2\frac{\sigma}{\sqrt{m-1}}\right)\]
which, together with~\eqref{eq:b-eps-ex}, yields the estimate on $a_*$ claimed in Example~\ref{ex:2}.

\bigskip
 
\noindent{\large{\bf Acknowledgements.}}  MA is supported by the ANR project DEEV,  ANR-20-CE40-0011-01, funded by the French Ministry of Research. PG is supported by the ANR project NOLO, ANR-20-CE40-0015, funded by the French Ministry of Research.


\bibliographystyle{abbrv}
\bibliography{biblio}

\end{document}